%% file: lch4sfs-11-2.tex
\documentclass[11pt]{amsart} 
\usepackage{graphicx}
\usepackage{rlepsf,epstopdf}
\usepackage{hyperref}
\include{header}

\def\dfn#1{{\textbf {#1}}}

\title{Legendrian Contact Homology in Seifert Fibered Spaces}

\author[J. Licata]{Joan E. Licata} \address{Stanford University,
  Stanford, CA 94305} \email{jelicata@stanford.edu}

\author[J. Sabloff]{Joshua M. Sabloff} \address{Haverford College,
Haverford, PA 19041} \email{jsabloff@haverford.edu} \thanks{JMS is
partially supported by NSF grant DMS-0909273.}

\date{\today}

\begin{document}

\begin{abstract}
We define a differential graded algebra associated to Legendrian knots in Seifert fibered spaces with transverse contact structures.  This construction is distinguished from other combinatorial realizations of contact homology invariants by the existence of orbifold points in the Reeb orbit space of the contact manifold.  These orbifold points are images of the exceptional fibers of the Seifert fibered manifold, and they play a key role in the definitions of the differential and the grading, as well as in the proof of invariance. We apply the invariant to distinguish Legendrian knots whose homology is torsion and whose underlying topological knot types are isotopic;  such examples exist in any sufficiently complicated contact Seifert fibered space.  \end{abstract}

\maketitle

\section{Introduction}\label{sec:intro}

In recent years, the study of Legendrian knots in contact manifolds besides the standard $\rr^3$ has attracted increasing attention.  One thread of this research has investigated the geography of Legendrian knots in rational homology $3$-spheres, while another thread has focused on extending non-classical invariants to an ever-wider range of contact $3$-manifolds. This paper unites these two themes with a study of Legendrian knots in closed orientable Seifert fibered spaces (SFS) over orientable bases equipped with transverse invariant contact structures, which we shall refer to as \dfn{contact Seifert fibered spaces}.  The main result is a combinatorial realization of the Legendrian contact homology differential graded algebra for Legendrian knots in contact SFS's.  

Legendrian knots in rational homology $3$-spheres have been studied from a variety of viewpoints.  Papers such as \cite{bg:leg-lens} and \cite{bgh:lens-HFK} develop combinatorial formulations of Heegaard Floer link homology for Legendrian links in lens spaces.  In work more closely related to this paper, Baker and Etnyre \cite{be:rational-tb} extend classical invariants of null-homologous Legendrian knots to \emph{rationally} null-homologous Legendrian knots and show that rational unknots in tight contact lens spaces are classified up to Legendrian isotopy by their rational classical invariants; see also \cite{bg:leg-lens, cornwell:lens-invts, ozturk}.  We continue Baker and Etnyre's inquiry by developing non-classical invariants for such knots.

Many non-classical invariants of Legendrian knots owe their genesis to geometric ideas of Eliashberg and Hofer \cite{yasha:icm}.  Legendrian contact homology is part of the Symplectic Field Theory framework \cite{egh}, and it was first rendered rigorous and combinatorially computable by Chekanov \cite{chv} for knots in the standard contact $\rr^3$; see also \cite{ens, lenny:computable}. The invariant takes the form of the homology of a non-commutative differential graded algebra $(\alg, \df)$.

Legendrian contact homology has been rigorously defined not only in the standard contact $\rr^3$, but also in the $1$-jet space of the circle \cite{lenny-lisa}; in higher-dimensional contact manifolds such as $\rr^{2n+1}$ \cite{ees:high-d-geometry} or $P \times \rr$, where $P$ is an exact symplectic manifold \cite{ees:pxr}; and in circle bundles over a Riemann surface whose contact structure is defined by a connection of negative curvature \cite{s1bundles}.  More recently, the first author has generalized the circle-bundle invariant to lens spaces with their universally tight contact structures \cite{joan:leg-in-lens}; see also \cite{joan:grid-1-comp}.  

The main goal of this paper is to further generalize the lens space invariant to Legendrian knots in closed oriented Seifert fibered spaces over orientable bases.   In order to accommodate this broader class of manifolds, we adopt the perspective that our Seifert fibered spaces are $S^1$ orbibundles over orientable $2$-orbifolds with cone singularities. Drawing on work of \cite{kt:cr-seifert, lisca-matic:transverse, lutz, massot}, we describe a positive $S^1$-invariant contact form whose Reeb orbits are the fibers of the Seifert fibered structure.  This encodes the Seifert fiber structure as a consequence of the contact geometry and permits us to prove the following:

\begin{thm} \label{thm:dga}
  Let $K$ be a Legendrian knot in a Seifert fibered contact manifold $(M, \alpha)$ as above.  Then there is a combinatorially defined differential graded algebra $(\alg, \df)$, called the \textbf{(low-energy) Legendrian contact homology differential algebra}, whose stable tame isomorphism type is invariant under Legendrian isotopy of $K$.  \end{thm}
 
The grading takes values in the rational numbers. The algebra may be defined over the coefficient ring $\zz_2[\qq]$, where the coefficients are represented by polynomials in rational powers of $t$. Setting $t=1$ yields an algebra graded by a cyclic group, and the resulting differential graded algebra can be used to distinguish Legendrian knots.

\begin{thm}\label{thm:ex}
  In any Seifert fibered contact manifold with $b>1$, at least one exceptional fiber, and for which the ratio of the orbifold Euler characteristic of the base to the Euler number of the whole space is non-integral, there exist Legendrian non-isotopic knots whose homology class is torsion and whose underlying topological knots are isotopic.
\end{thm}

The constructions used to define the grading for rationally null-homologous knots have additional applications, and we develop combinatorial algorithms for computing the rational classical invariants in the related paper \cite{ls:tb-paper}. We conjecture that these algorithms strengthen Theorem~\ref{thm:ex} with the additional hypothesis that the knot type is non-simple.

The remainder of the paper is structured as follows: Section~\ref{sec:geometry} provides background on Seifert fibered spaces and transverse invariant contact structures on them.  We also introduce orbifolds and orbibundles as technical tools for later use. In Section~\ref{sec:knots-in-sfs}, we turn our attention to Legendrian knots in these Seifert fibered spaces, developing machinery that will allow us to define the differential graded algebra $(\alg, \partial)$ in Section~\ref{sec:dga}.  We postpone the proofs of $\partial^2=0$ and of invariance to Sections~\ref{sec:d2}  and \ref{sec:invariance}, respectively; these two sections are rather technical, and they focus on the differences between the current invariant and other combinatorial realizations of Legendrian contact homology.  In Section~\ref{sec:ex}, we present families of examples which prove Theorem~\ref{thm:ex}.

\subsection*{Acknowledgements}

We thank Emmanuel Giroux for useful discussions about invariant contact forms on Seifert fibered spaces.  We also thank MSRI for hosting the second author during part of the research process.

\section{Transverse Invariant Contact Forms on Seifert Fibered Spaces}
\label{sec:geometry}

We assume familiarity with standard definitions from contact geometry: positive contact structure, Reeb vector field, Legendrian knot, and Lagrangian projection. An introduction to the relevant background material may be found in \cite{etnyre:intro} and \cite{geiges:intro}.  After introducing Seifert fibered spaces, we will discuss the interactions between their contact geometry and fiber structure. 

\subsection{Invariant Transverse Contact Structures on Seifert fibered
  spaces}
\label{ssec:transverse-structure}

A Seifert fibered space is a closed, connected three-manifold together with a decomposition as a disjoint union of circles, called \dfn{fibers}. Each fiber is required to have a fibered neighborhood of a special type.  Begin with a solid torus $D^2\times S^2$ whose fibers are $\{ pt\} \times S^1$.  Cutting this solid torus along $D^2\times \{pt\}$ and regluing via the map $(r, \theta)\mapsto (r, \theta+\frac{p}{q})$ induces a new fiber structure on the neighborhood of the core fiber $\{0\}\times  S^1$.  When $\frac{p}{q}\in \mathbb{Z}$, the core fiber is called a \dfn{regular} fiber, and  when $\frac{p}{q}\in \mathbb{Q}\setminus \mathbb{Z}$, the core is called an \dfn{exceptional} fiber.  Let $\Sigma$ be the quotient of $M$ by identifying each fiber to a point. Then the quotient map $M\rightarrow \Sigma$ induces an orbifold structure on $\Sigma$, a perspective we will explore more fully in Section~\ref{ssec:orbibundle}.  The image of each exceptional fiber is an \dfn{exceptional point} on the orbifold $\Sigma$.  Note that we may also think of a Seifert fibered space as a $3$-manifold with a semi-free $S^1$-action.

 In order to describe these manifolds, we follow the notational conventions of  \cite{lisca-matic:transverse, massot}: 
let $S$ be a closed, oriented surface of genus $g$.  Consider $r+1$ disjoint discs $D_0, \ldots, D_r$, and let $S' = S \setminus \bigcup \mathring{D}_i$, with each boundary component oriented as the boundary of the removed disc.  Let $M' =S' \times S^1$.  The first homology groups of the boundary tori of $M'$ are generated by classes $\langle m_i, \ell_i \rangle$, with $m_i= [\partial D_i \times \{pt\}]$ and $\ell_i = [\{pt\} \times S^1]$, oriented so that $m_i \cdot \ell_i = 1$. 

Let $b\in \mathbb{Z}$, and glue a solid torus to the boundary component  $\partial D_0 \times S^1$ so that a meridian  is sent to a curve representing the class of  $m_0 + b \ell_0$.  For $1\leq i\leq r$, let  $\alpha_i$ and $\beta_i$ be relatively prime integers such that $0 < \beta_i <\alpha_i$. Glue a solid torus $W_i$ to the $i^{th}$ boundary torus of $M'$ so that a meridian of $W_i$ is sent to a simple closed curve representing the homology class $\alpha_i m_i+ \beta_i \ell_i$. The resulting identification space is the Seifert fibered space with Seifert invariants $(g,b; (\alpha_1, \beta_1), \ldots, (\alpha_r, \beta_r))$; the number $b$ is the \dfn{integral Euler number}.
Note that the first homology of this Seifert fibered space  is generated by the first homology of $S'$, the classes of the curves $m_i$, and the class $F$ of a regular fiber.

Every Seifert fibered space can be realized via this construction, and given two Seifert invariants, it is easy to determine whether they correspond to the same Seifert fibered manifold \cite{orlik}. The \dfn{rational Euler number} of a Seifert fibered space with Seifert invariants $(g,b; (\alpha_1, \beta_1), \ldots, (\alpha_r, \beta_r))$ is the rational number
\[e(M)=-b-\sum_{i=1}^r \frac{\beta_i}{\alpha_i} .\]

Throughout this paper, we will restrict attention to the case where $\Sigma$ and $M$ are both orientable.  We are interested in $S^1$-invariant transverse contact structures on these Seifert fibered spaces, and Kamishima and Tsuboi (and also Lisca and Matic) have determined when such a contact structure exists. With the notation introduced above, their theorem is the following:

\begin{thm}[\cite{kt:cr-seifert, lisca-matic:transverse}]
  On a Seifert fibered space, there exists an $S^1$-invariant transverse contact form if and only if the rational Euler number is negative.
\end{thm}
    
In fact, we can adjust the contact form promised by the theorem so that its Reeb vector field is particularly nice:

\begin{lem} \label{lem:good-reeb}
  Any $S^1$-invariant transverse contact form on a Seifert fibered space is contactomorphic to a contact form whose Reeb vector field points along the fibers. \end{lem}

\begin{proof}
  Let $\alpha$ be the invariant transverse contact form, and let $X$
  be the vector field that generates the circle action on $M$. Since
  $\alpha$ is positive and $\ker \alpha$ is transverse to the fibers,
  the function $\alpha(X)$ is invariant and always positive. Thus, we
  may form $\overline{\alpha} = \frac{1}{\alpha(X)} \alpha$.  Notice that
 \begin{equation*}
    \iota_X d \overline{\alpha} = \mathcal{L}_X \overline{\alpha} - d \iota_X
    \overline{\alpha} = 0;
  \end{equation*}
  the Lie derivative term vanishes because $\overline{\alpha}$ is invariant and $\iota_X\overline{\alpha}$ is constant.  Thus, we have $X \in \ker d\overline{\alpha}$, and hence that the Reeb field points along the fibers.
\end{proof}

As a result of this lemma, we will assume that the phrase ``contact Seifert fibered space'' refers to an orientable Seifert fibered space over an orientable base, together with a contact form whose Reeb trajectories realize the Seifert fibers. 

Finally, in the case that $\Sigma$ is not simply-connected, we equip $\Sigma$ with a collection of oriented simple closed curves $\{X_i\}$ that form a basis for $H_1(\Sigma)$.  We denote the union of these curves by $\mathbf{X}$.

\begin{lem} After possibly adjusting the contact structure locally, each $X_i$ may be chosen to be the Lagrangian projection of a Legendrian loop in $M$. \end{lem}

\begin{proof} Starting at an arbitrary point on $X_i$, lift the entire curve to a Legendrian curve in $M$.  In general, this curve will not be closed, and we adjust the contact form on the bundle over a neighborhood of $X_i$ so that the endpoints agree.  (A similar argument is used in \cite{s1bundles}.) Let $\beta\in \Omega^1(\Sigma)$ be an $S^1$-invariant form representing the Poincar\'{e} dual of $X_i$.  Scale $\beta$ so that with respect to the contact form $\alpha+\pi^*k\beta$, the curve $X_i$ lifts to a closed Legendrian curve in $M$. Use a partition of unity to ensure a smooth transition from $\alpha$ to the adjusted local form.
\end{proof}

\subsection{Seifert Fibered Spaces as Orbibundles}
\label{ssec:orbibundle}

As mentioned in the introduction, we will make use of the language of orbifolds and orbibundles in our analysis. We use the terminology of \cite{cr:orbi-gw} to discuss orbifolds and orbibundles. The key facts, all easily verified, are presented here.

The surface $\Sigma$ has a natural orbifold structure induced by the Lagrangian projection of $M$.  In particular, the image of the $j^{th}$ exceptional fiber has a local $\zz/\alpha_j$ action.  Its orbifold Euler characteristic is given by
\begin{equation} \label{eq:orb-euler-char}
  \chi_{orb}(\Sigma) = \chi(\Sigma) + \sum_{j=1}^r \left( \frac{1}{\alpha_j} - 1 \right).
\end{equation}

The manifold $M$ has the structure of an $S^1$-orbibundle over $\Sigma$.  A fiber neighborhood of the $j^{th}$ exceptional fiber has a diagonal action of $\zz/\alpha_j$ given by
\[(z,t) \mapsto \bigl( \exp (\frac{2\pi
  i}{\alpha_j})z, \exp (\frac{2\pi
  i \beta_j}{\alpha_j}) t \bigr).\]
The manifold $M$ may be thought of as the unit circle orbibundle of a complex line orbibundle.  Its orbifold Euler number (or first Chern number) is the rational Euler number defined above.

A smooth map of orbifolds $u: X\to X'$ is called \dfn{regular} if the
preimage of the regular points of $X'$ is an open, dense, and
connected subset of $X$. It follows from Lemmas 4.4.3 and 4.4.11 of \cite{cr:orbi-gw} that if $E$ is an orbibundle over $X'$ and $u:X\to X'$ is a regular map, then there is a canonical pull-back orbibundle $u^*\pi: u^*E\to
X$ and a smooth map $\tilde{u}:u^*E\to E$ which covers $u$.

In the specific case of the orbibundle $\pi: M\to \Sigma$, the contact form $\alpha$ induces a form $\Omega$ on $\Sigma$ which satisfies $\pi^*\Omega=d\alpha$. We define the Euler curvature form on $\Sigma$ to be $\kappa = -\frac{1}{2\pi}\Omega$. The usual Chern-Weil theory, adapted to orbibundles, then allows us to compute:
\begin{equation} \label{eqn:total-euler}
  e(M) = \int_\Sigma \kappa, 
\end{equation}
where integration is performed in the orbifold sense.  As in \cite{lutz}, we combine this calculation  with the requirement that $e(M)<0$ in order to view $\kappa$ as a negative multiple of the volume form on $\Sigma$.

\section{Legendrian Knots in Seifert Fibered Spaces}
\label{sec:knots-in-sfs}

In this section, we turn our attention to Legendrian knots in the contact manifolds described above.  This section provides the key link between the geometric behavior of these knots and the combinatorics of the algebra defined in Section~\ref{sec:dga}, and more generally, it establishes a scheme for for representing Legendrian knots in contact Seifert fibered spaces diagrammatically.

\subsection{Lagrangian Diagrams for Legendrian Knots}
\label{ssec:diagrams}

Let $K$ be a generic Legendrian knot in a transverse contact Seifert fibered space $M$; assume that $K$ lies in the complement of the exceptional fibers.  Let $\pi:M\rightarrow \Sigma$ denote Lagrangian projection, and let $\Gamma=\pi(K)$.  The preimage of each double point of $\Gamma$ consists of a primitive Reeb orbit which intersects $K$ twice, partitioning the orbit into two Reeb chords that each begin and end on $K$.

Let $\ell(x)$ denote the length of the Reeb chord $x$:
\[ \ell(x)=\int_x\alpha.\]
A regular fiber has length $2\pi$, and we will sometimes refer to chords with length less than the orbital period as being ``short''.  

Label the double points of $\Gamma$ with $\{1, \ldots, n\}$.  Each crossing locally divides $\Sigma$ into four quadrants, and orienting $K$ distinguishes a pair of quadrants whose boundaries are oriented coherently.  At the $i^{th}$ crossing, label this pair of quadrants by $a_i^+$ and $b_i^-$ as in Figure~\ref{fig:local-labels}.  Label the other pair of quadrants by $a_i^-$ and $b_i^+$.  

\begin{figure}[ht]
  \begin{center}
    \scalebox{.5}{\includegraphics{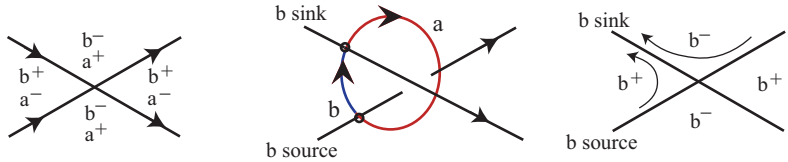}}
 \end{center}
 \caption{Quadrant labels near a double point of $\Gamma$.}
 \label{fig:local-labels}
\end{figure}

These labels correspond to the short Reeb chords projecting to the crossing in the following manner:  each oriented chord identifies the strands of $K$ locally as ``source'' and ``sink''; traveling from the $b$ source strand to the $b$ sink strand in $\Gamma$ orients the quadrant labeled $b^+$ positively.  Similarly, traveling from the $a$ source strand to the $a$ sink strand in $\Gamma$ orients the quadrant labeled $a^+$ positively. We will use the terms \dfn{$a$-type} and \dfn{$b$-type} to designate chords labeled with $a_i$ and $b_j$, respectively, and we say that the $a$-type chords are \dfn{preferred}.

A \dfn{Lagrangian diagram} for the pair $(M, K)$ consists of a pair $(\Sigma, \Gamma)$ which is diffeomorphic to the Lagrangian projection of $(M,K)$, together with  markings indicating the exceptional points and their Seifert invariants. When $\pi_1(\Sigma)\neq 1$, we also assume that $\Gamma$ and the $X_i$ intersect transversely.

A Lagrangian diagram is \dfn{labeled} if it is decorated with the following additional data:
\begin{enumerate}
\item an orientation of $\Gamma$;
\item chord labels $a_i$ and $b_i$ at each crossing of $\Gamma$; 
\item for each component $R \in \Sigma \setminus  \big(\Gamma \cup \mathbf{X})$,  a rational number $n(R)$, called the \dfn{defect}, which will be defined in the next section.
\end{enumerate}

\subsection{Admissible Disks}
\label{ssec:adm}

By itself, the isotopy class of $\Gamma$  is insufficient to recover the Legendrian type of $K$.  This is remedied by labeling each region in the Lagrangian diagram with its \dfn{defect}, a notion which extends the one introduced in \cite{s1bundles} to the case of $S^1$-orbibundles.  Instead of working directly with regions, however, we use the somewhat more general concept of an admissible disc.

\begin{defn} 
  A \dfn{marked surface} $\dd$ is a complex orbifold of real dimension $2$ with $m+1$ distinguished points $z_0, \ldots, z_m$ on $\partial \dd$ and $l$ distinguished points $w_1, \ldots, w_l$ in the interior of $\dd$ ($m,l\geq 0$). The points $z_0, \ldots, z_m$ are called \dfn{marked boundary points}, and the points $w_1, \ldots, w_l$ are called \dfn{marked interior points}. The singular points of $\dd$ are contained in the set of marked interior points and all have cyclic local groups.

  When $\dd$ is topologically the unit disc in the complex plane, we refer to the \dfn{marked disc}.  
\end{defn}

\begin{defn}\label{defn:adm} A map $u: (\dd, \partial \dd) \to (\Sigma, \Gamma \cup \mathbf{X} )$ of a marked surface into a Lagrangian diagram is \dfn{admissible} if the following conditions are satisfied:
  \begin{enumerate}
  \item away from the marked points, $u$ is a smooth, orientation-preserving immersion, as is the restriction  of $u$ to $\partial \dd \setminus \{z_0, \ldots z_m \} $;
  \item The preimage of each exceptional point in $\Sigma$ is a marked interior point of $\dd$.
  \item If an exceptional point of $\Sigma$ has invariants $(\alpha_i, \beta_i)$, then  the associated marked interior point is an orbifold point with group $\zz_{k_i}$ for some $k_i$ dividing $\alpha_i$. In a neighborhood of the associated marked interior point, $u$ is modeled on $z \mapsto z^{\frac{\alpha_i}{k_i}}$.
   \item Each marked boundary point $z_i$ maps to a double point of $\Gamma \cup \mathbf{X}$, and the $u$-image of a neighborhood of $z_i$ covers an odd number of quadrants of $\Sigma$ near the double point.
   \end{enumerate}
\end{defn}

Two admissible maps $u$ and $v$ are deemed to be \dfn{equivalent} if
there is an automorphism $\phi: \dd \to \dd$ of marked surfaces such
that $u = v \circ \phi$.

An admissible map $u$ is a regular map in the orbifold sense, so we may pull back the orbibundle $M \to \Sigma$ to an orbibundle over $\dd$.  Note that if $k_i=1$ at each exceptional point, then this pullback will be an honest circle bundle over a Riemann surface.

Recall that the Lagrangian projection from $M$ to $\Sigma$ is a bundle map, and let $\tilde{u}: u^*M \to M$ be the map covering $u$.  Let $\mathbf{x} = \{x_0, \ldots, x_m\}$ be a choice of short Reeb chords lying above the images under $u$ of the marked boundary points.  The image of $u$ covers a quadrant labeled with $x_i^+$ or $x_i^-$ near the image of the marked boundary point, as indicated in  Figure~\ref{fig:local-labels}; let $\epsilon_i$ be $+1$ (resp.\ $-1$) if $x_i$ is traversed in the positive (resp.\ negative) direction.

\begin{defn} \label{defn:defect} With $u$ as above, the \dfn{defect} of $u$ with respect  to the chords $\mathbf{x}$ is defined by:
  \begin{equation} \label{eq:defect}
    n(u;\mathbf{x}) = \int_{\dd} u^*\kappa + \frac{1}{2\pi} \sum_{i=0}^m \epsilon_i \ell(x_i).
  \end{equation}
\end{defn}

The defect allows us to label the Lagrangian diagram for $(M,K)$.
For any region $R\subset \Sigma\setminus \big( \Gamma \cup \mathbf{X} \big)$, choose an admissible map
$u:\dd\rightarrow \Sigma$ with image $\overline{R}$ and define the
\dfn{defect of $R$} as the defect of $u$ with the choice of the preferred $a$-type chords at the corners. The definition of the defect of a region, together with Equation~(\ref{eqn:total-euler}), implies:

\begin{prop}\label{prop:defe} Let $\{R_j\}_{j=1}^r$ be the regions of $\Sigma\setminus \big( \Gamma \cup \mathbf{X}\big)$.  Then 
 \[ \sum_{j=0}^r n(R_j)=e(M).\]
\end{prop}

\subsection{The Meaning of the Defect}
\label{ssec:defectprop}

Admissible maps without singular points in their domains  will play a key role in defining the differential for the Legendrian contact homology algebra. 
In order to better understand the defect of an admissible map with a given choice of Reeb chords $\mathbf{x}$, we construct a curve $C_{\mathbf{x}}$ in $u^*M$ as follows: lift $\partial \dd$ to a curve $C_\mathbf{x}: S^1 \to u^*M$ which consists of segments of the Legendrian lifts of $\partial \dd$, alternating with the preferred chords $x_i$ traversed in the direction determined by $u$.  If $\tilde{u}: u^*M \to M$ is the bundle map covering $u$, then it is easy to see that:
\begin{equation} \label{eq:cx}
  \int_{C_{\mathbf{x}}} \tilde{u}^*\alpha = \sum_i \epsilon_i \ell(x_i).
\end{equation}

The defect of an admissible map $u$ with no singular points in its domain is the obstruction to extending the lifted boundary curve $C_{\mathbf{x}}$ to a map from a disc to $M$.

\begin{prop}\label{prop:defect-obstruction}
  Suppose that $u$ is an admissible map with no singular points in its domain and that $\mathbf{x}$ is a choice of Reeb chords over the corners of $u$.  Then $\tilde{u}|_{C_{\mathbf{x}}}$ extends to a map from a disc to $M$ if and only if $n(u;\mathbf{x}) = 0$.
\end{prop}

\begin{proof}
  We begin by pulling back the orbibundle $M$ over the admissible map $u$ to obtain an honest circle bundle $u^*M \stackrel{u^*\pi}{\to} \dd$.  Trivialize $u^*M$ using the bundle equivalence $\tau: u^*M \to D^2 \times S^1$, and let $p_2: D^2 \times S^1 \to S^1$ be the projection onto the second factor.  This allows us to view $p_2 \circ \tau \circ C_{\mathbf{x}}$ as a self-map of $S^1$; abusing terminology, we call the degree of this map $\deg C_{\mathbf{x}}$.  We will prove that 
  \begin{equation} \label{eq:defect-deg}
    n(u; \mathbf{x}) = \deg C_{\mathbf{x}},
  \end{equation}
  and the proposition will follow.

Form a new curve $C'_{\mathbf{x}}$ by concatenating $C_{\mathbf{x}}$ and a vertical curve that winds $-\deg C_{\mathbf{x}}$ times around the fiber.  Using a small vertical homotopy, we may perturb $C'_\mathbf{x}$ to be a section $s$ of $u^*M$ over $\partial \dd$; note that the integral of $\alpha$ over the images of these two curves is the same.  The degree (as defined above) of the section $s$ vanishes, so it extends to a section $\hat{s}: \dd \to u^*M$.  In particular, we have the relation $(u^*\pi) \circ \hat{s} = id$ in the following commutative diagram:

\begin{equation*}
  \xymatrix{
u^*M \ar[r]^{\tilde{u}} \ar[d]_{{u}^*\pi} & M \ar[d]^{\pi}  \\  
\dd \ar[r]_{{u}}  \ar@<-1ex>[u]_{\hat{s}}  & \Sigma}
\end{equation*}

We now calculate as follows:
\begin{align*}
  \int_\dd u^*\Omega &= \int_\dd \hat{s}^* (u^*\pi)^* u^*\Omega \\
  &= \int_\dd \hat{s}^* \tilde{u}^* d\alpha \\
  &= \int_{\partial \dd} \hat{s}^* \tilde{u}^* \alpha \\
  &= \int_{C_{\mathbf{x}}}  \tilde{u}^* \alpha - 2\pi \deg C_{\mathbf{x}}.
\end{align*}
Dividing by $2\pi$, recalling that $\kappa = -\frac{1}{2\pi} \Omega$, using Equation (\ref{eq:cx}) and rearranging yields (\ref{eq:defect-deg}).
\end{proof}

\begin{rem}\label{rem:homdef}
In fact, the proof of Proposition~\ref{prop:defect-obstruction} supports the following more general statement: for any admissible map whose domain has a single boundary component and no exceptional points, the first homology of the bundle $u^*M|_{\partial \dd}$ has a natural basis represented by a copy of the fiber $F$ and curve which bounds a section in $\dd \times S^1$.  The defect of $u$ is the coefficient of $[F]$ when $[C_{\mathbf{x}}]$ is expressed in terms of this basis.  When the image of $u$ contains exceptional points, we pass to an $a$-fold cover, where $a$ is the least common multiple of the $\alpha_i$ indices of the exceptional points.   In this case, the homology class of the lift of $C_{\mathbf{x}}$ determines  $(a)\big( n(u; \mathbf{x})\big)$.\end{rem}

\begin{cor} \label{cor:defect-properties}
  The defect of an admissible map $u$ without singular points in its domain is an integer strictly bounded above by the number of Reeb chords traversed in the positive direction.  Furthermore, we have the bound
  $$\sum_{x_i \text{ positive}} \ell(x_i) - \sum_{x_j \text{ negative}}  \ell(x_j) > 2\pi n(u; \mathbf{x}).$$
\end{cor}

\begin{proof}
  That the defect is an integer follows from its characterization as the degree of a map in the preceding proof.  The upper bound comes from the fact that only positive terms in the expression (\ref{eq:defect}) for the defect come from the positively traversed chords, and each of these chords has length at most $2\pi$.  
\end{proof}

\subsection{Rotation of an Admissible Map}
\label{ssec:rotation}

In the previous sections, we assigned a defect to each admissible map $u$ (assuming a choice of Reeb chords at the corners), which measures the relative Euler number of the bundle $u^*M$ relative to a curve formed by Legendrian lifts of the boundary of the base and the chosen Reeb chords.  In this section we assign another relative Euler number to an admissible map, namely the rotation of a tangent vector field to $\Gamma \cup \mathbf{X}$ along the boundary $\partial \dd$. 

We begin by examining a version of the Poincar\'e-Hopf theorem for $2$-orbifolds with boundary.

\begin{lem} \label{lem:orbi-obstruction}
  The obstruction to extending a smooth, non-vanishing vector field tangent to the boundary of a $2$-orbifold $Q$ to the interior is given by
  $$r(Q) = \chi_{orb}(Q).$$
\end{lem}

\begin{proof}
 The geometry behind this lemma comes from the Poincar\'e-Hopf Theorem for orbifolds with boundary, reduced to the $2$-dimensional case; see \cite{seaton:two-thms}.  For a $2$-orbifold  $Q$ with boundary and  a vector field $Y$ whose zeroes are isolated and lie in the interior of $Q$, the orbifold Poincar\'e-Hopf Theorem states
\begin{equation} \label{eq:ph}
  \ind_{orb}(Y) = \chi'_{orb}(Q) + \bar{Y}^*\Upsilon[\partial Q];
\end{equation}
we explain the constituent terms below.

The term $\ind_{orb}(Y)$ on the left is the sum of the  orbifold indices of $Y$ at its zeroes, which is precisely the obstruction to $Y$ being a non-vanishing vector field.  

The term $\chi'_{orb}(Q)$ is the so-called \dfn{inner orbifold Euler characteristic} of $Q$.  The ordinary orbifold Euler characteristic of Equation (\ref{eq:orb-euler-char}) may be computed using a triangulation $\mathcal{T}$ of the orbifold with the property that the order of the isotropy group of the local action is constant on the interior of every simplex; see \cite{moerdijk-pronk} for a proof that such a triangulation exists.  If we let $N_\sigma$ denote the order of the isotropy group for the simplex $\sigma$, then we may compute that
$$\chi_{orb}(Q) = \sum_{\sigma \in \mathcal{T}} \frac{(-1)^{\dim \sigma}}{N_\sigma}.$$
The inner Euler characteristic is defined similarly, except the sum ignores the simplices that are completely contained in the boundary.
It is easy to check, however, that $\chi'_{orb}(Q) = \chi_{orb}(Q)$ in two dimensions.

The final ingredient of the Poincar\'e-Hopf Theorem is an orbifold version of Sha's secondary Euler class \cite{sha:secondary-euler}, denoted $\Upsilon$, that lies in the cohomology of the unit sphere bundle of $TQ|_{\partial Q}$. In the spirit of Chern, this class is defined using connection and curvature forms for an $SO(n)$ connection on the tangent bundle.  The actual term that appears is the evaluation of the pullback of $\Upsilon$ by the vector field $Y$ (thought of as a section $\bar{Y}$ of the unit sphere bundle) on the fundamental class of $\partial Q$.  In the case that $Y$ is tangent to the boundary, however, Sha notes that $\bar{Y}^*\Upsilon$ vanishes; see also \cite{seaton:thesis}.
\end{proof}

As above, let $u: \dd \to \Sigma$ be an admissible map, and let $Z$ be a smooth non-vanishing vector field on $\partial \dd$ that is tangent to $\partial \dd$.  On $\partial \dd \setminus \{z_0, \ldots, z_n\}$, we have $u_*Z \in T(\Gamma \cup \mathbf{X})$, and hence the rotation of $Z$ measures the rotation of a vector field tangent to $\Gamma \cup \mathbf{X}$ along the image of $u$.  The only issue occurs at the boundary marked points $z_i$.  We make corrections at those points as follows: once we have fixed a complex structure on $\Sigma$ in a neighborhood of each double point of $\Gamma$ with the property that the strands of $\Gamma$ intersect orthogonally, we see that near each $z_i$, $u_*Z$ has a discontinuity which is an integer multiple of $\pi/2$.  To pin down this multiple, we examine the number $m_i$ of quadrants of $\big(\Sigma, \Gamma \cup \mathbf{X} \big)$ that the image of $u$ covers near $z_i$; the multiple may be easily computed to be $\frac{1}{4}(m_i-2)$.  We define the \dfn{rotation number} $r(u)$ of an admissible map $u$ to be the obstruction to extending the vector field $Z$ to a vector field on all of $\dd$, corrected by the corner terms $\sum_{i=0}^m \frac{1}{4}(m_i - 2)$; reversing orientation negates the rotation number.  Thus, based on Lemma~\ref{lem:orbi-obstruction}, we have the following computation:
 \begin{equation} \label{eq:rotation}
    r(u) = \epsilon_u \left( \chi_{orb}(\dd) + \sum_{i=0}^m \frac{1}{4}(m_i - 2) \right),
  \end{equation}
  where $\epsilon_u = \pm 1$ according to whether $u$ preserves or reverses orientation, and $m_i$ is the number of quadrants of $\big(\Sigma, \Gamma \cup \mathbf{X} \big)$ that the image of $u$ covers near $z_i$.

As with the defect, we may assign a rotation to each region $R\in \Sigma\setminus \ \big( \Gamma  \cup \mathbf{X}\big)$.  Let $u_R:\dd\rightarrow R$ be an admissible map with the property that the pre-image of each exceptional point with Seifert invariants $(\alpha_i, \beta_i)$ is an orbifold point of order $\alpha_i$.  Then the rotation of $R$ is the rotation of the map $u_R$.

\section{The Differential Graded Algebra}
\label{sec:dga}
In this section we define the differential algebra $\alg$ associated to labeled Lagrangian diagram for a Legendrian knot in a contact Seifert fibered space.

\subsection{The Algebra}
\label{sec:alg}

Given a Legendrian knot $K$ in a contact Seifert fibered space $M$, we consider an algebra generated by the short Reeb chords of the pair $(M,K)$.  We define the \dfn{(low-energy) algebra} of the labeled Lagrangian projection $\Gamma$ with $n$ double points:

\begin{defn} \label{defn:alg}
  Let $\alg_\Gamma$ be the semi-free unital associative algebra with  coefficients in $\zz_2[\qq]$ generated by the letters $\{a_1, b_1,  \ldots, a_n, b_n\}$.
\end{defn}

\subsection{The Grading}
\label{ssec:grading}

In order to assign a grading to each generator of $\alg_\Gamma$, we first introduce the notion of a \dfn{formal capping surface}.  

A formal capping surface is a vector in $\mathbb{Z}^{|\Sigma\setminus \Gamma|}$, or, equivalently, an assignment of an integer to each region of $\Sigma\setminus \Gamma $.  The integer assigned to the region $R$ is called the \dfn{multiplicity} of $R$ in the formal capping surface.

\begin{defn}If $S=(c_1, c_2, ... c_n)$ is a formal capping surface, the \dfn{defect} of $S$ is the sum of the defects of the regions $R_j$, weighted by multiplicity:
\[ n(S)=\sum_j c_j n(R_j).\]  
The \dfn{rotation} of $S$ is the sum of the rotations of the regions $R_j$, weighted by multiplicity:
\[ r(S)=\sum_j c_j r(R_j).\]
\end{defn}


\subsubsection{Grading generators}\label{sec:capp} 

In this section we  assign a formal capping surface to each short Reeb chord.  Although this depends on a number of choices, we will show in Section~\ref{ssec:choices} that different sets of choices produce isomorphic differential graded algebras. 

Given a generator $x$, choose an immersed loop $\gamma\subset \Gamma$ which is smooth away from a corner at the crossing associated to $x$. If $x$ is an $a$-type generator, orient $\gamma$ to positively bound a quadrant labeled $x^+$ near the corner.  If $x$ is a $b$-type generator,  orient $\gamma$ to negatively bound a quadrant labeled $x^-$.  

For clarity of exposition, we begin by considering the case where $\Sigma$ is simply connected.  The idea is to apply a variant of Seifert's algorithm to $\gamma$ in order to construct a formal capping surface. Resolve each crossing of $\gamma$ so that the result is a disjoint union $\coprod \gamma_j$ of oriented simple closed curves with corners.  For each $j$, choose a component of $\Sigma\setminus \gamma_j$.  This \dfn{Seifert region} is a union of regions in $\Sigma\setminus \Gamma$.  If the orientation of $\gamma_j$ agrees with the boundary orientation of the Seifert region, assign each constituent region a multiplicity of $1$ and the complementary regions a multiplicity of $0$.  If the boundary orientation agrees with that of $-\gamma_j$, assign each constituent region a multiplicity of $-1$ and the complementary regions a multiplicity of $0$.  In each region, sum the contributions coming from each of the $\gamma_j$.  The resulting multiplicity in each region defines a formal capping surface which we denote $S_x$.

\begin{figure}[ht]
\begin{center}
\scalebox{.6}{\includegraphics{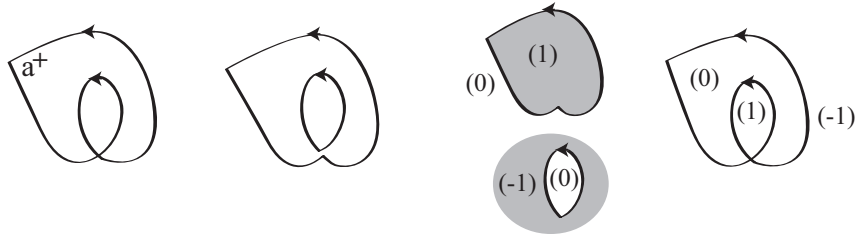}}
\caption{Constructing a formal capping surface $S_a=(1,0,-1)$.}
\end{center}
\end{figure} 

\begin{defn}\label{def:gengr}  The grading of the generator $x$ with formal capping surface $S_x$ is defined by
\[ |a|=2\big(r(S_a)+\mu n(S_a)-\frac{1}{4}\big)\]
\[|b|=2\big(r(S_b)+\mu n(S_b)+\mu-\frac{1}{4}\big),\] 
where $\mu=-\frac{\chi_{orb}(M)}{e(M)}$.
\end{defn}

From a computational perspective, we need only choose one capping surface per double point.  In particular, given a capping surface $S_a$, we may choose the capping surface $S_b$ to be the orientation reverse of $S_a$.  With this choice, we have
$$|b| = 2 \mu -1 -|a|.$$

When $\Sigma$ has non-trivial first homology, the resolution of the capping path $\gamma$ may result in simple closed curves $\gamma_j$ which are non-separating on $\Sigma$.  In this case, we note that each $\gamma_j$ is homologous to a unique linear combination $\sum_j a_{i_j}X_i$.  In order to assign a formal capping surface to $\gamma_j$, include $-a_{i_j}$ parallel copies of $X_i$ into the diagram and resolve any new double points as above.  The resulting collection of disjoint oriented simple closed curves is null-homologous, and hence partitions $\Sigma$ into positively and negatively bounded subsurfaces.  One may then choose, for example, the union of the positively bounded subsurfaces, though there are other possible choices here. As in the simply connected case, these subsurfaces assign to each region of $\Sigma\setminus \big(\Gamma \cup \mathbf{X} \big)$ a multiplicity of $1$, $0$, or $-1$, thus assigning a formal capping surface to the capping path $\gamma_j$.

\begin{rem}
  In \cite{egh}, the recipe for defining a grading in an SFT-type theory requires a choice of trivialization of the contact structure $\xi$ over a basis for $H_1(M)$.  Though it appears that we have made no such choice above, the choices are implicit in our construction.  The lifts of the $X_i$, together with the regular and exceptional fibers, generate $H_1(M)$.  We have trivialized $\xi$ over the (Legendrian) lift of each $X_i$ using the tangent vectors to that lift.  Similarly, the choice of $\mu$ implicitly trivializes $\xi$ over the regular and exceptional fibers.
\end{rem}

\subsection{The Differential}
\label{ssec:diff}

The definition of the differential is similar in philosophy to that of Chekanov's combinatorial theory for Legendrian contact homology in the standard contact $\rr^3$ \cite{chv} and builds on the extensions in \cite{joan:leg-in-lens, s1bundles}.  As before, the differential will be defined by a count of discs in $\Sigma$ whose boundaries lies on $\Gamma$, with marked points on the boundary of the disc mapping to double points of $\Gamma$.  As in \cite{s1bundles}, the topology of the space $M$ must be taken into account via the defects, but the new feature in the current situation is that the discs need not --- in fact,
\emph{cannot} --- be immersed at the projections of the exceptional fibers.

Recall the definition of an admissible map (Definition~\ref{defn:adm}).
 
\begin{defn} \label{defn:disc} Given labels $x$ and $y_1, \ldots, y_k$ taken from the generators $\{a_i, b_i\}$, the set $\Delta(x;  y_1, \ldots, y_k)$ consists of all admissible maps of marked discs $u$, up to  equivalence, such that
  \begin{enumerate}
  \item the map $u$ has no singular points in its domain;
  \item the map $u$ sends a neighborhood of $z_0$ to a quadrant
    labeled $x^+$ in $\Sigma$; and
  \item the map $u$ sends a neighborhood of the marked boundary point $z_i$, $i>0$, to a quadrant labeled $y_i^-$ in $\Sigma$.
  \end{enumerate}
\end{defn}

Suppose that $u\in \Delta(x;  y_1, \ldots, y_k)$.  Let $S_{j}$ denote a capping surface for $y_{i_j}$, and let $S_u$ be the formal capping surface whose multiplicity at a non-exceptional point $c$ is $|u^{-1}(c)|$.  In each region of $\Sigma\setminus ( \Gamma \cup \mathbf{X})$, sum the multiplicities of the formal capping surfaces $S_u$, $\{S_j \}_{j>0}$, and $-S_x$. This sum defines a new formal capping surface $S$, and we write
\[ |u|=  r(S)+\mu n(S).\] 

\begin{defn}\label{defn:diff}
  The \dfn{differential} $\df: \alg_{\Gamma} \to \alg_{\Gamma}$ is defined on the generators  by
  \begin{equation} \label{eq:d} 
  \df x = \sum_{\mathbf{y}}  \sum_{\substack{[u] \in \Delta(x,\mathbf{y}) \\ n(u; x,
        \mathbf{y}) = 0}} y_{1} \cdots y_{k} t^{-|u|}.
  \end{equation}
  \end{defn}

As in Definition~\ref{defn:diff}, we extend to $\mathcal{A}_{\Gamma}$ using linearity and the Leibniz rule.  

Define the grading of the coefficient $t^q$ by
\[ |t^q|=q.\]

The following special case of Corollary~\ref{cor:defect-properties} shows that the differential \df\ is well-defined (that is, the sum in (\ref{eq:d}) is finite).

\begin{lem} \label{lem:stokes}If $[u]\in \Delta(x,\mathbf{y})$ represents a summand of $\df x$, then 
\[ \ell(x)-\sum_i \ell(y_i)>0. \]
  \end{lem}
    
Definition~\ref{defn:diff} makes $(\alg, \df)$ into a differential graded algebra (DGA):

\begin{thm} \label{thm:d2} 
  The differential \df\ has degree $-1$ and satisfies $\df^2=0$.  
\end{thm}

Further, this differential algebra gives rise to a Legendrian invariant.  The precise statement of invariance involves the algebraic notion of ``stable tame isomorphism''; see \cite{chv, ens}.  For now,  note that stable tame isomorphism implies quasi-isomorphism.

\begin{thm} \label{thm:invariance}
  The stable tame isomorphism type of $(\alg, \df)$ is invariant under Legendrian isotopies of $K$.  \end{thm}

Of course, this means that the standard objects derived from the differential algebra are also invariants: the augmentation number \cite{ns:augm-rulings},
the set of linearized (co)homologies \cite{chv} and their product structures \cite{products}, the characteristic algebra \cite{lenny:computable}, etc.

These two theorems combine to prove Theorem~\ref{thm:dga}; their proofs will appear in Sections~\ref{sec:d2} and \ref{sec:invariance}, respectively.


\subsection{Independence of choices}
\label{ssec:choices}

The definition of the grading and the upcoming arguments in Section~\ref{sec:invariance} depend heavily on the formal capping surface assigned to each Reeb chord.  We show here that the algebras associated to different choices of formal capping surface are isomorphic as differential graded algebras.

\begin{lem}\label{lem:surfacechoice} If $x$ is a generator with fixed capping path $\gamma$, the grading of $x$ is independent of the choice of formal capping surface compatible with $\gamma$.
\end{lem}

It follows that the grading of each generator depends only on the choice of capping path.  Deferring the proof of the lemma for the moment, we will now show that the graded algebras associated to different capping paths are nevertheless isomorphic. More precisely, suppose that $\sigma$ and $\tau$ are two functions which assign gradings to generators by associating a capping path and compatible formal capping surface to each Reeb chord.  Denote the associated differential graded algebras by $\mathcal{A}_{\tau}$ and $\mathcal{A}_{\sigma}$, respectively.  

\begin{thm} $\mathcal{A}_{\tau}$ and $\mathcal{A}_{\sigma}$ are tame isomorphic as differential graded algebras.  
\end{thm}

\begin{proof}  Define $\phi:\mathcal{A}_{\tau} \rightarrow \mathcal{A}_{\sigma}$ on generators by \[ \phi(x)=t^{\tau(x)-\sigma(x)}x.\]

One may easily verify that $\partial_{\sigma}\circ \phi=\phi\circ \partial_{\tau}$. \end{proof}

\begin{rem} \label{rem:meaning-of-iso}
  We can, in fact, identify the quantity $\tau(x)-\sigma(x)$ in the preceding proof.  For a fixed Reeb chord $x$, the difference between the capping paths associated  to $\sigma$ and $\tau$ is some signed number $k_x$ of copies of  $\Gamma$.  Let $S$ be a capping surface for $\Gamma$, constructed as in Section~\ref{sec:capp}.  Then we have:
  $$\tau(x) - \sigma(x) = k_x \bigl(2r(S) + 2\mu n(S) \bigr).$$
\end{rem}

\begin{proof}[Proof of Lemma~\ref{lem:surfacechoice}] Given a capping path $\gamma$ for $x$, recall that we construct a formal capping surface by resolving the double points of $\gamma$ and then assigning to each of the resulting simple closed curves either a positive or a negative subsurface of $\Sigma$. We will show that the grading of $x$ does not depend on these choices.  

For simplicity, let $\gamma$ denote a single separating component of the resolved capping path. Let $S_P$ denote the region of $\Sigma$ bounded positively by $\gamma$ and let $S_N$ denote the region bounded negatively by $\gamma$. 
Since the union of  $S_N$ and  $S_P$ is $\Sigma$,  Proposition~\ref{prop:defe} implies that the sum of the defects of $N$ and $P$ is the rational Euler number of $M$:
\[ n(S_N)+n(S_P)=e(M).\]
Furthermore, we claim that
\[ r(S_P)-r(S_N)=\chi_{orb}(\Sigma).\]
To see this, note that the orbifold Euler characteristic is additive under identification of two 2-orbifolds along the boundary and that at a convex (resp.\ concave) corner of $S_P$, the correction term to $r(S_P)$ is $\frac{1}{4}$ (resp.\ $-\frac{1}{4}$), while the correction coming from the complementary concave (resp.\ convex) $r(S_N)$ is $-\frac{1}{4}$ (resp.\ $\frac{1}{4}$).

The closed curve coming from $\gamma$ bounds $S_N$ negatively, so we have the following computation:

\begin{align*}
  r(S_N)+\mu n(S_N)&=r(S_P)-\chi_{orb}(\Sigma) +(\frac{\chi_{orb}(\Sigma)}{-e(M)})(-1)\big(e(M)-n(S_P)\big)\\
  &=r(S_P)-\chi_{orb}i(\Sigma)+\chi_{orb}(\Sigma)+\mu n(S_P)\\
  &=r(S_P)+\mu n(S_N).
\end{align*}

A similar argument applies to a non-separating component of the capping path when considering the choice between the positive and negative subsurfaces as at the end of Section~\ref{sec:capp}.
\end{proof}

Finally, we note that from a computational perspective, it is often convenient to work over $\mathbb{Z}_2$ instead of $\mathbb{Z}_2[\mathbb{Q}]$.  As in \cite{chv, ens, s1bundles}, this can be achieved by setting the $\mathbb{Q}$ parameter $t$ equal to $1$; in this case, the grading is only well-defined modulo $2r(S) + 2\mu n(S)$, where $S$ is a formal capping surface whose boundary is $\Gamma$, as in Remark~\ref{rem:meaning-of-iso}.

\section{Examples and Applications}\label{sec:ex}

In this section, we present a collection of examples related to computing the differential algebra and distinguishing Legendrian knot types.

\subsection{A first example}

As a first example, we consider the knot shown in Figure~\ref{fig:basic}.
\begin{figure}[ht]
  \begin{center}
    \relabelbox \footnotesize{
      \centerline{\scalebox{.9}{\epsfbox{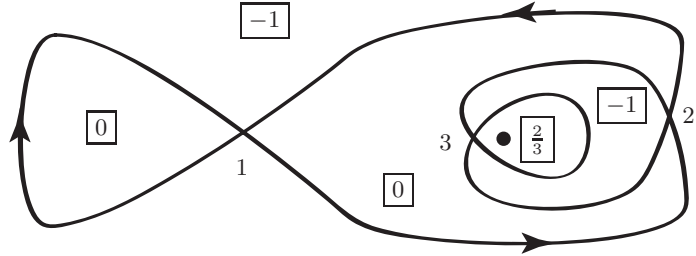}}}}
    \relabel{0}{\fbox{$0$}}
    \relabel{1}{\fbox{$-1$}}
    \relabel{2}{\fbox{$0$}}
    \relabel{3}{\fbox{$-1$}}
    \relabel{4}{\fbox{$\frac{2}{3}$}}
    \relabel{5}{$1$}
    \relabel{6}{$2$}
    \relabel{7}{$3$}
    \endrelabelbox
    \caption{A Legendrian knot $K$ in $L(4,3)$. }
    \label{fig:basic}
  \end{center}
\end{figure}

We begin by computing the gradings of the generators, choosing the capping path suggested by the orientation for each $a_i$ and reversing its orientation to get a capping path for the associated $b_i$.  Note that $\mu=1$.  Then for all $i$, we have
\[
  |a_i| =1 \text{ \quad  and \quad } |b_i| = 0.\]

The generator $a_1$ admits two boundary discs, each of which corresponds to one of the two possible capping paths for $a_1$.  Neither path has corners, so the boundary word associated to each of these discs is $t^{-|u|}$.  In each of the terms above, the power of $t$ is $0$.  To see this, note that for each $u$ which contributes to the boundary, the concatenation of $u(\partial \mathbb{D})$ with the chosen capping path either contracts to a point in $\Gamma$ or is $\Gamma$ itself.  The former case always results in $t^0$; in the latter case, the power of $t$ is $0$ because the contributions to $|u|$ from the left and right sides of the central double point cancel.  Thus, we obtain $\df a_1 = 0$.

\begin{figure}[ht]
\begin{center}
\scalebox{.6}{\includegraphics{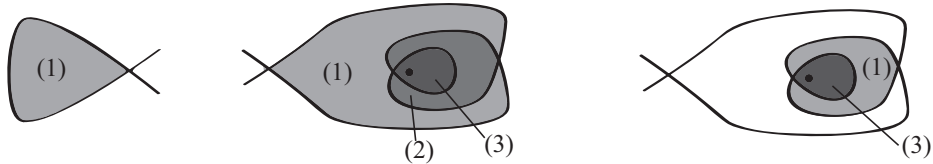}}
\caption{The three boundary discs, shown with multiplicities in each region.}\label{fig:basicexdiscs}
\end{center}
\end{figure}

There is one other map of the disc into the diagram which contributes boundary terms.  The image of this disc has multiplicity one in the region with defect $-1$ and multiplicity three in the region with defect $\frac{2}{3}$. This disc yields the terms
\begin{align*}
 \partial a_2&= b_3 \\
 \partial a_3&= b_2
 \end{align*}
In this case, the argument that the $t$ power is zero for both terms is the same as the one above.

Note that, in general, every admissible defect $0$ map of a marked disc with $k$ marked boundary points that map to convex corners will contribute $k$ terms to the differential, one term for the positive chord at each corner.

\subsection{Torsion Knots}
\label{ssec:torsion-knots}

\begin{figure}
  \relabelbox \footnotesize{
    \centerline{\epsfbox{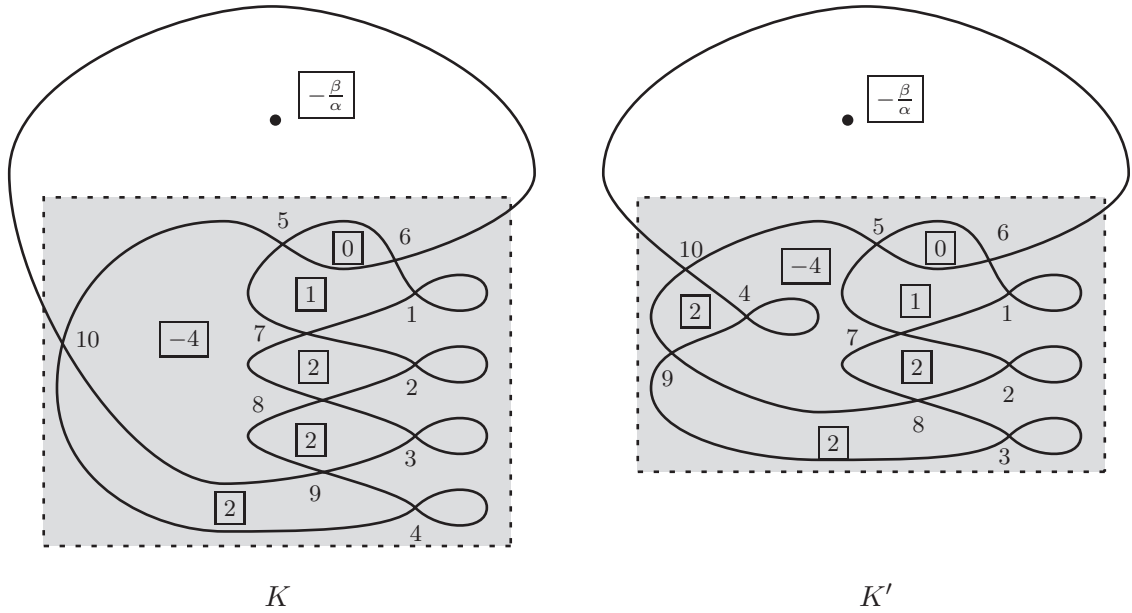}}}
  \relabel{0}{\fbox{$-\frac{\beta}{\alpha}$}}
  \relabel{1}{\fbox{$-4$}}
  \relabel{2}{\fbox{$0$}}
  \relabel{3}{\fbox{$1$}}
  \relabel{4}{\fbox{$2$}}
  \relabel{5}{\fbox{$2$}}
  \relabel{6}{\fbox{$2$}}
  \relabel{7}{\fbox{$-\frac{\beta}{\alpha}$}}
  \relabel{8}{\fbox{$0$}}
  \relabel{9}{\fbox{$1$}}
  \relabel{a}{\fbox{$2$}}
  \relabel{b}{\fbox{$2$}}
  \relabel{c}{\fbox{$2$}}
  \relabel{d}{\fbox{$-4$}}
  \relabel{e}{$1$}
  \relabel{f}{$2$}
  \relabel{g}{$3$}
  \relabel{h}{$4$}
  \relabel{i}{$5$}
  \relabel{j}{$6$}
  \relabel{k}{$7$}
  \relabel{l}{$8$}
  \relabel{m}{$9$}
  \relabel{n}{$10$}
  \relabel{o}{$1$}
  \relabel{p}{$2$}
  \relabel{q}{$3$}
  \relabel{r}{$4$}
  \relabel{s}{$5$}
  \relabel{t}{$6$}
  \relabel{u}{$7$}
  \relabel{v}{$8$}
  \relabel{w}{$9$}
  \relabel{x}{$10$}
  \relabel{y}{\normalsize{$K$}}
  \relabel{z}{\normalsize{$K'$}}
  \endrelabelbox
  \caption{Diagrams for two non-Legendrian-isotopic knots $K$ and $K'$ representing the torsion homology class $[m_i-F]$. The defect of each of the small lobes is $0$.  We assume that the loop at the top either encircles an exceptional fiber or has defect at most $-1$.  Note that most of this knot lives in a vertically thin slice of the SFS, and that with respect to the short chords, all defects pictured are $0$ except for the region containing the exceptional fiber.}
  \label{fig:torsion-knot}
\end{figure}

In this section, we present an example of a pair of topologically isotopic Legendrian knots distinguished by their Legendrian contact homology algebras.  This construction provides the proof of Theorem~\ref{thm:ex}, and the two non-isotopic Legendrian representatives appear in Figure~\ref{fig:torsion-knot}.  

These knots may be constructed as follows:  Fix any Seifert fibered space $M$ with at least one exceptional fiber, $b>1$, and $\mu \not \in \zz$.  Recall that the construction described in Section~\ref{ssec:transverse-structure} builds a Seifert fibered space from an $S^1$-bundle by Dehn surgery.  Perform all but the $i^{th}$ Dehn surgery, and call the resulting contact Seifert fibered space $M_0$.  In $\Sigma_0$, consider an embedded circle that bounds a disc containing the marked point corresponding to the final Dehn surgery, but no other exceptional points.  Increase the radius of this circle until its Legendrian lift $K_0$ is a closed curve homotopic to a regular fiber. The defect of the region bounded by the circle is $-1$; since we have assumed $b>1$, such a region exists!

In a small ball $B$ around a point on $K_0$, we may assume that the pair $(B, K_0 \cap B)$ is contactomorphic to the standard contact $\rr^3$ paired with the $x$ axis.  Let $K_1$ and $K_1'$ be the long Legendrian knots whose Lagrangian projections are pictured inside the dashed boxes in Figure~\ref{fig:torsion-knot}.  Modify $K_0$ by replacing $K_0 \cap B$ with the image of each of $K_1$ and $K_1'$.  

Now perform the final surgery as indicated above.  This results in the knots $K, K'\subset M$. Note that the defect of the region containing the exceptional point is $-\frac{\beta_i}{\alpha_i}$; this follows from Remark~\ref{rem:homdef}. The homology class of each of $K$ and $K'$ is $[m_i-F]$ in $M$.

To show that $K$ and $K'$ are not Legendrian isotopic, we apply Chekanov's technique of linearized homology \cite{chv} to the low-energy DGA.  Further, we use the ground ring $\zz_2$ instead of $\zz_2[\qq]$.   As noted at the end of Section~\ref{ssec:choices}, the gradings are then well-defined modulo $\frac{2}{\alpha} - 2\mu (1+\frac{\beta}{\alpha})$. We begin the computation by finding the gradings of the generators. 

\begin{align*}
  |a_{1\ldots 4}| &= 1 & |b_{1 \ldots 4}| &= 2\mu-2\\
  |a_5| &= -3 & |b_5| &= 2\mu+2 \\
  |a_6| &= 3 & |b_6| &= 2\mu-4 \\
  |a_{7 \ldots 10}| &= 2 \mu -1& |b_{7 \ldots 10}| &= 0
\end{align*}

With the knot $K'$, the gradings are the same as before, except for the following:

\begin{align*}
  |a'_{5}| &= -1 & |b'_{5}| &= 2\mu\\
  |a'_{6}| &= 1& |b'_{6}| &= 2\mu-2
\end{align*}

The next step is to find the augmentations of each DGA, i.e., graded algebra maps $\epsilon: \alg \to \zz_2$ that vanish on the image of $\df$.  To check for augmentations, it suffices to look at the differentials of the generators of degree $1$.  For $K$, we obtain

\begin{align*}
  \df a_1 &= 1+ b_7 + a_6 a_5 b_7  &
  \df a_2 &= 1+ b_7 b_8 \\
  \df a_3 &= 1+ b_8 b_9 &
  \df a_4 &= 1+ b_9 b_{10}
\end{align*}

The differentials for $K'$ are the same up to reordering, and in addition, $\df a'_6 = 0$. In both cases, there is a unique augmentation $\epsilon$ that sends $b_{7 \ldots 10}$ to $1$ and all other generators to $0$.

Finally, we linearize the differentials.   Define $\phi^\epsilon: \alg \to \alg$ by $\phi^\epsilon(x) = x + \epsilon(x)$, then conjugate by $\phi^\epsilon$ and take the linear terms. One can easily check that $a_6$ cannot appear in the linearized differential for degree reasons, so the linearized homology of $\alg_K$ with respect to $\epsilon$ is nontrivial in degree $3$; in contrast, the linearized homology of $\alg_{K'}$ with respect to $\epsilon$ must be trivial in degree $3$, as there are no generators in that degree. Since, as proved by Chekanov \cite{chv}, the set of all linearized homologies taken over all possible augmentations is invariant under stable tame isomorphism, we may conclude that $K$ and $K'$ are not Legendrian isotopic.

\begin{rem}
  In fact, one can compute that the sets of Poincar\'e-Chekanov polynomials for $K$ and $K'$ are $\{t^{-3} + t^3 + t^{2\mu-4} + t^{2\mu+2} \}$ and $\{t^{-1} + t + t^{2\mu} + t^{2\mu-2}\}$, respectively.  Unlike the case of Legendrian knots in $\rr^3$, the linearized homologies of these knots do not have a ``fundamental class'' in degree $1$ as promised by \cite{duality}.  The duality structure of the $\rr^3$ case (again, see \cite{duality}) does appear to persist in some form, though we conjecture that this easy appearance of duality is an artifact of the construction of the knots $K$ and $K'$ and that the general duality structure --- if it exists at all --- is more subtle.
\end{rem}

\section{Proof that $\df$ is a differential}
\label{sec:d2}

\subsection{Proof that $\df$ is graded with degree $-1$. } 

Recall that each disc counted by the differential defines a formal capping surface $S_u$ whose multiplicity is $|u^{-1}(c)|$ at a regular point $c \in \Sigma \setminus (\Gamma \cup \mathbf{X})$. 

Suppose that $u$ is an admissible map representing the term $y_1...y_kt^{-|u|}$ in $\partial x$, and suppose further that each of $x$ and the $y_j$ are $a$-type generators. Since defect and rotation are additive, we may compute $|u|$ as a sum of the contributions from $S_u$ and the capping surfaces associated to the generators. Let $S_j$ denote a capping surface for $y_j$, $j\geq 1$, and let $S_0$ denote a capping surface for $x$.  Then

\[|u|=|S_u|-|S_0|+\sum_{j=1}^k |S_j|.\]

Expanding this, we have

\begin{equation}\label{eq:line1}|u|=2\big[r(S_u)+\mu n(S_u)-r(S_0)-\mu n(S_0)+  \sum_{j=1}^k\big( r(S_j)+\mu n(S_j)\big)\big].\end{equation}

The fact that $u$ defines a term in $\df x$ implies that $S_u$ has defect zero and rotation $1-\frac{k+1}{4}$. Combining this with Equation~(\ref{eq:line1}) yields:
$$-|t^{-u}|=1-|x|+\sum_{j=1}^k |y_{i_j}|.$$
This shows that the differential is a graded map when restricted to the subalgebra generated by the $a$-type chords.

Now consider how the computations above change if some $a_{j}$ is replaced by the generator $b_{j}$ whose capping path is $-\gamma_j$  for $1\leq j \leq k$.  In this case, the fact that the defect of $u$ vanishes implies that $n(S_u) =1$ in (\ref{eq:line1}).  However, since
\[|b_{j}|=2\big[r(S_j)+\mu n(S_j)+\mu-\frac{1}{4}\big],\]
the conclusion again follows.  The argument for the $j=0$ case is similar. 

\subsection{Proof that $\df^2=0$ }

The proof that $\df$ satisfies $\df^2=0$ will be a combinatorial realization of the standard compactness and gluing arguments in Morse-Witten-Floer theory.  The spirit of the combinatorics goes back to Chekanov \cite{chv}; the formulation of the proof mirrors that in \cite{s1bundles}.

We will work with two new types of disc: first, a \dfn{broken disc} is a pair $(u,v)$ with $u \in \Delta(x; y_1, \ldots, y_k)$, $v \in \Delta(y_i; w_1, \ldots, w_l)$ for some $1 \leq i \leq k$, and where $n(u;x,\mathbf{y}) = 0 = n(v;y_i, \mathbf{w})$.  Every term in $\df^2x$ corresponds to a broken disc. Second, an \dfn{obtuse disc} is an admissible disc $u: (\dd, \partial \dd) \to (\Sigma, \Gamma)$ that satisfies nearly all the conditions for inclusion in some $\Delta(x;\mathbf{y})$ except that near one marked boundary point $z_i$, the image of $u$ covers three quadrants of $(\Sigma, \Gamma)$ instead of just one.  We again require that  $n(u;x,\mathbf{y})=0$.  The definition of the set $\Delta(x; \mathbf{y})$ can be adjusted to include obtuse discs if we specify the following: 
\begin{enumerate} 
\item If $z_0$ is the obtuse point, then the two of the three quadrants covered by the image of $u$ are labeled $x^+$ and
\item If $z_i$, $i>0$, is the obtuse point, then the two of the three quadrants covered by the image of $u$ are labeled $y_i^-$.
\end{enumerate}

Theorem~\ref{thm:d2} will follow from two lemmata:

\begin{lem}[Compactness]
  \label{lem:compactness} Every obtuse disc splits into a broken disc in exactly two distinct ways.
\end{lem}

\begin{lem}[Gluing]
  \label{lem:gluing} Every broken disc can be glued uniquely to form a obtuse disc.
\end{lem} 
  
\begin{proof}[Proof of Lemma~\ref{lem:compactness}]
Let $u$ be an obtuse disc with its obtuse corner at the image of $z_i \in \partial \dd$.  At the obtuse corner, there are two line segments of $\Gamma$ that start at $u(z_i)$ and go into the interior of the image of $u$. For each line segment, construct a path $c:[0,1] \to \dd$ that begins at $z_i$, whose image under $u$ follows $\Gamma$, and that ends the first time one of the following occurs:
\begin{enumerate}
\item The path intersects $\partial \dd$ at an unmarked point;
\item The path intersects itself; or
\item The path intersects $\partial \dd$ at $z_i$.
\end{enumerate}
The image of $c$ divides $\dd$ into two marked subdiscs $\dd_1$ and $\dd_2$, and $u(c(1))$ must lie at a double point of $\Gamma$. Place additional marked points 
 in $\partial \dd_i$, $i=1,2$, so that $c(1)$ is marked in both discs.  Let $v_i = u|_{\dd_i}$ for $i=1,2$.  The marked points in the domain of $v_i$ inherit labels from $u$, except at the new marked points. Choose labels at the double point $u(c(1))$ so that each $v_i$ has exactly one positive label; note that these labels will both be $a$ labels or they will both be $b$ labels.

To finish the proof of Lemma~\ref{lem:compactness}, it remains to show that the pair $(v_1,v_2)$ is a broken disc.  As all of the admissibility and $\Delta$ conditions are local, it is obvious from the construction above that both $v_1$ and $v_2$ are admissible discs and that the positive corner of $v_2$ lies adjacent to a negative corner of $v_1$. It remains to show that both $v_1$ and $v_2$ have vanishing defect.  The definition of the defect implies that $n(v_1) + n(v_2) = n(u) = 0$, and the proof now follows from Corollary~\ref{cor:defect-properties}.
\end{proof}

\begin{proof}[Proof of Lemma~\ref{lem:gluing}]
Let $(u_1,u_2)$ be a broken disc.  Since the images of the $u_j$ cover adjacent quadrants at the corner labeled $y_i$, those images share an edge of $\Gamma$.  More precisely, there exist immersed paths $c_j: [0,1] \to \partial \dd_j$, $j=1,2$, so that $u_1(c_1(t)) = u_2(c_2(t))$ and $u_1(c_1(0)) = y_i = u_2(c_2(0))$, and the images are maximal among such paths.  Glue the domains of $u_1$ and $u_2$ along the images of the $c_j$, removing the marked points $c_j(0)$.  We will show below that it cannot be the case that both $c_1(1)$ and $c_2(1)$ are marked points.  Define $v$ on the new domain by patching together $u_1$ and $u_2$ and smoothing if necessary.

In order to prove Lemma~\ref{lem:gluing}, we need to show that $v$ is an obtuse disc.  The only nontrivial facts to prove are
\begin{enumerate}
\item the defect of $v$, with respect to the Reeb chords it inherited from the $u_j$, vanishes; and
\item $v$ has the proper behavior at the marked points lying at  the endpoints of the image of $u_j \circ c_j$.
\end{enumerate}

For the first item, note that the defect of $v$ is the sum of the defects of the $u_j$ since the integrals of the Euler curvature are additive and the chords eliminated from the $u_j$ have equal and opposite signed lengths.  The fact that the defects of the $u_j$ vanish now implies that the defect of $v$ vanishes as well.  

For the second item, consider the common corner of the $u_j$.  There, $v$ does not actually have a corner, as the associated marked points have been removed from the domain of $v$.  Thus, we need only show that $c_1(1)$ is a marked point in the domain of $v$ with the property that $v(\mathbb{D})$  covers three quadrants near $v\big(c_1(1)\big)$.  There are three possibilities, which we evaluate in turn:
\begin{enumerate}
\item It cannot be the case that neither $c_1(1)$ nor $c_2(1)$ is a marked point, or else the paths would not be suitably maximal.  
\item Nor can it be the case that both $c_1(1)$ and $c_2(1)$ are marked points.  Suppose, to the contrary, that this were the case.  Since $v$ has a positive marked point at $c_2(0)$, $c_2(1)$ cannot be positive.  

  If $c_1(1)$ were a positive marked point, then we could replace $v$ by an admissible disc $v': \dd' \to \Sigma$ such that $v = v'$ and $\dd'$ is the same as $\dd$ except for a lack of marked points at the point corresponding to $c_j(1)$.  This new admissible disc would have only negative corners and no singular points, and hence a negative defect --- but the defect of $v'$ is zero, a contradiction.

If both $c_1(1)$ and $c_2(1)$ are negative marked points, say at $y_l$ and $w_m$, then notice that $y_l$ and $w_m$ must be labels for complementary chords. Summing the equations for the defects of $u_1$ and $u_2$ yields
$$0 = \ell(x) - \sum_{j \neq i} \ell(y_j) - \sum_j \ell(w_j) + \int_\dd (u_1^*\kappa + u_2^*\kappa).$$
Notice that all terms except $\ell(x)$ are negative.  Further, we know that $\ell(x) < 2\pi$, whereas $\ell(y_l) + \ell(w_m) = 2\pi$. This is a contradiction, so this case, too, is impossible.
\item The remaining case permits exactly one of  $c_1(1)$ or $c_2(1)$ to be a marked point, and in this case we obtain an obtuse corner for $v$.
\end{enumerate}
\end{proof}

\section{Proof of Invariance}
\label{sec:invariance}

In this section we prove that Legendrian isotopic knots have stable tame isomorphic differential graded algebras.  We refer the reader to \cite{chv, ens} for the definition of stable tame isomorphism, with the caveat that we also allow elementary isomorphisms of the form $\phi(x) = t^{\alpha_x} x$
for each generator $x$ of $\alg$ as in \cite{lenny:lsft}.

Throughout this section, we will employ the following notation: let $K^-$ denote a Legendrian knot in $M$, and let $K^+$ denote the knot which results from applying an isotopy to $K^-$.  The Lagrangian projections and algebras of these knots are correspondingly denoted by $\Gamma^{\pm}$ and $(\mathcal{A}^{\pm}, \partial^{\pm})$.  We may assume that the combinatorics of $\Gamma^\pm$ differ only locally.

If the isotopy region lies in the complement of the exceptional fibers, then $\Gamma^-$ and $\Gamma^+$ will differ by a surface isotopy in $\Sigma$ (which does not change the isomorphism class of the algebra), or by a Legendrian Reidemeister move.   If the isotopy passes a segment of $K^-$ across an exceptional fiber with Seifert invariants $(\alpha, \beta)$, then $\Gamma^+$ will differ from $\Gamma^-$ by an $\alpha$-teardrop move, as depicted in Figure~\ref{fig:fcsisotopy} for $\alpha=4$.  Label the new crossings in $\Gamma^+$ from $1$ to $\alpha-1$, so that the index increases with proximity to the branch point.  As a first step in proving invariance, we study the relationship between the gradings on $\mathcal{A}^-$ and $\mathcal{A}^+$.

\subsection{Invariance and the grading}\label{sec:grinv}

If each generator of $\alg^-$ is assigned a formal capping surface, then the isotopy induces a canonical collection of formal capping surfaces for the generators of $\alg^+$.  In the case of isotopies that create new intersections, we will impose a further compatibility condition on the formal capping surfaces chosen.  These conventions are necessary for the proof that the stable tame isomorphism induced by a local isotopy is graded.  

Given a formal capping surface for $x$ on $(\Sigma, \Gamma^-)$, we construct a formal capping surface on $(\Sigma,\Gamma^+)$ which preserves multiplicities in regions away from the isotopy. Near the isotopy region, the new formal capping surface is determined by the condition that crossing a strand of $\Gamma$ effects the same change in multiplicity before and after the isotopy.  (Note that this crossing is oriented.)

\begin{figure}[ht]
\begin{center}
\scalebox{.5}{\includegraphics{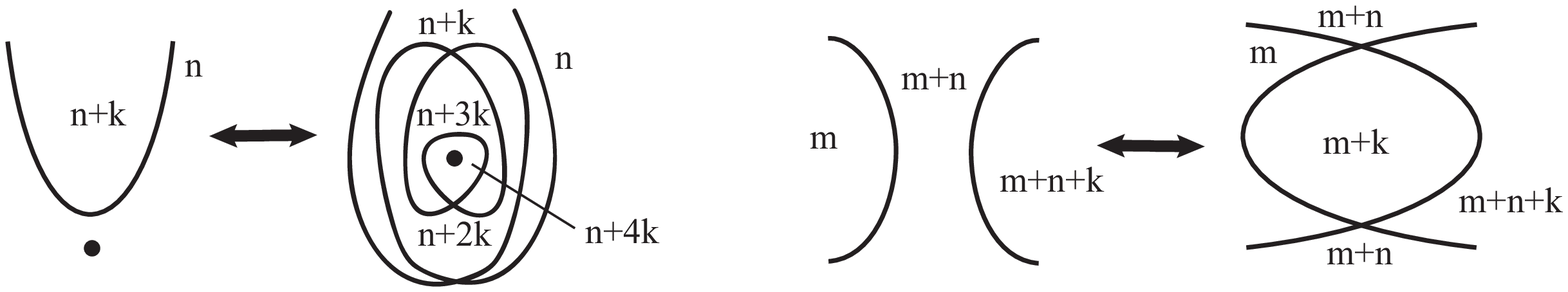}}
\caption{ The integers represent multiplicities of a formal capping surface before and after an isotopy. In the isotopy region, preserve the change in multiplicity associated to crossing each strand of $\Gamma$.}\label{fig:fcsisotopy}
\end{center}
\end{figure}

\begin{prop}\label{prop:isotinv} Let $S^-$ be a capping surface for the Reeb chord $x$, and let  $S^+$ be the image of $S^-$ under an $\alpha$-teardrop move. Denote the grading of $x$ with respect to the surface $S^*$ by $|x|^*$.  Then $|x|^-=|x|^+$.
\end{prop}

\begin{proof} 
Suppose that the formal capping surface $S\subset (\Sigma, \Gamma^-)$ had multiplicity $0$ at the exceptional point and multiplicity $1$ in the adjacent region $R$.  (See Figure~\ref{fig:fcsisotopy} with $n=0$, $k=1$.) After the isotopy, the local diagram has regions $\{R_i\}_{i=0}^\alpha$, where the multiplicity of the capping surface in $R_{i}$ is equal to $i$.  We can compute the rotation of each of these as in Section~\ref{ssec:rotation}:
\begin{center}
\begin{tabular}{r|c|c|c|c}
  & $R_1$ & $R_{1<i<\alpha-1}$ & $R_{\alpha-1}$ & $R_\alpha$ \\ \hline
  $r$ & $r(R)-\frac{3}{4}$ & $0$ & $\frac{1}{4}$ & $\frac{1}{\alpha} - \frac{1}{4}$
\end{tabular}
\end{center}

Scaling these values by the associated multiplicity, the local contribution to the rotation of the new formal capping surface is
\begin{align}
\sum_k k\,r(R_k)&=\big(r(R)-\frac{3}{4}\big) +  (\alpha-1) \big(\frac{1}{4}\big)+ \alpha \big(\frac{-1}{4}+\frac{1}{\alpha}\big)\notag\\
&=r(R).\notag
\end{align}

Since rotation of regions is additive, this result holds for all possible initial multiplicities.  Furthermore, the fact that the isotopy passes $K$ across a disc implies that the defects of the formal capping surfaces before and after the isotopy are the same.
\end{proof}

To prove that Legendrian Reidemeister moves preserve grading, consider the stable tame isomorphisms defined in \cite{s1bundles}.  It is straightforward to verify that with the formal capping surfaces as described above, these maps are graded.

\subsection{Isotopy invariance}\label{sect:isotinv}
 
 As above, suppose that $K^-$ and $K^+$ differ by a Legendrian isotopy in $M$. In the case when $\Gamma^-$ and $\Gamma^+$  differ by a Legendrian Reidemeister move, the proof of invariance follows largely from \cite{s1bundles}.  We therefore  restrict attention to the teardrop move which results from  passing a segment of $K^-$ across an exceptional fiber with Seifert invariants $(\alpha, \beta)$.  Away from a neighborhood of the exceptional point $e$, the surface isotopy class of $\Gamma^-$ is preserved, and we may similarly assume all labels on $\Gamma^+$ and $\Gamma^-$ in the complement of the isotopy region agree. 

The proof that $(\mathcal{A}^+,\partial^+)$ and $(\mathcal{A}^-, \partial^-)$ are equivalent differential algebras proceeds in five steps.  First, we stabilize $(\mathcal{A}^-, \partial^-)$ and denote the resulting algebra  by $(\mathcal{A}', \partial')$. Second, we study the relationship between $\partial^+$ and $\partial'$, and in the third step we use this data to define a tame isomorphism $s:(\mathcal{A}^+, \partial^+)\to(\mathcal{A}', \partial')$.  In the fourth step, we show  that $\partial'=s\circ \partial^+\circ s^{-1}$ on the subalgebra corresponding to the generators of $\mathcal{A}^-$.  Finally, we define a tame automorphism $g:(\mathcal{A}', \partial')\to (\mathcal{A}', \partial')$  so that $g\circ s\circ \partial^+\circ s^{-1}\circ g^{-1}=\partial'$.

\subsection{Step 1: Constructing $\mathcal{A}'$}\label{sec:step1}

Let $\mathcal{A}'$ denote the $(\alpha-1)$-fold stabilization of $\mathcal{A}^-$, where $\mathcal{E}_i=\{ e_1^i, e_2^i\}$:
\[ \mathcal{A}'=\mathcal{A}^- \coprod \mathcal{E}_1\coprod \mathcal{E}_2 ...\coprod \mathcal{E}_{\alpha-1}. \] 
Extend the length function on $\mathcal{A}^-$ to $\mathcal{A}'$ by assigning the stabilizing generators have lengths which correspond to those coming from the new crossings in $\Gamma^+$:
\begin{align}
\ell(e_{1}^k)&=\ell(a_{k})\notag\\
\ell(e_{2}^k)&=\ell(b_{\alpha-k})\text{ for  }1\leq k \leq \alpha-1.\notag
\end{align}

Next, we will assign gradings to the stabilizing generators of $\mathcal{A}'$. For each $k$, there exists a map $u_k \in \Delta(a_k; b_{\alpha-k})$; see Figure~\ref{fig:brbigon}. 

Construct a  formal capping surface for $a_k$ using the loop which runs along the $u_k$-image of $\partial \dd$.  For $b_{\alpha-k}^-$, choose the loop which is a subset of the $a_k$ loop.  Choose the formal capping surface $S_{b_{\alpha-k}}$ which differs from $S_{a_k}$ only by the image of $u_k$.

Assign gradings to the stabilizing generators of $\mathcal{A}'$ so that for $1\leq k \leq \alpha-1$,
\begin{align}
|e_{1}^k|&=|a_{k}|\notag\\
|e_{2}^k|&=|b_{\alpha-k}|.\notag
\end{align}

\begin{lem}\label{lem:abdefect}
With $u_k$ as above,  $n(u_k; a_k, b_{\alpha-k})=0$.
 \end{lem}

\begin{proof}

Since $u_k\in \Delta(a_k, b_{\alpha-k})$, it follows that $n(u_k; a_k, b_{\alpha-k})$ is integral. Restricting the isotopy region sufficiently  renders the curvature term in the defect negligible, so the defect depends entirely on the lengths of the chords. 

We may assume that the isotoped strand remains in a sufficiently small ball centered on the exceptional fiber.  Thus, the lengths of any new chords created by the isotopy are arbitrarily close to integral multiples of the length of the exceptional fiber: \[\ell(a_k), \ell(b_{\alpha-l}) \in \left\{ \frac{2\pi}{\alpha}, \frac{4\pi}{\alpha}, ... \frac{2(\alpha-1)\pi}{\alpha} \right\}.\] The contributions of the indicated chords to the defect have opposite signs, so the only possible integral value for $n(u; a_k, b_{\alpha-k})$ is $0$.
\end{proof}

\begin{figure}[ht]
\begin{center}
\scalebox{.6}{\includegraphics{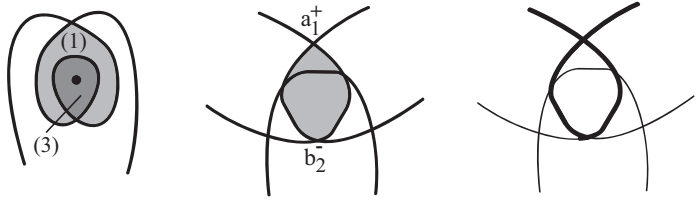}}
\caption{ The left  figure indicates the multiplicities associated to the image of $u_1$, which sends the two marked boundary points to  double points labeled $a_1$ and $b_2$.  The center figure indicates the image of the same map lifted to the $3$-fold branched cover of the teardrop region.  The right figure indicates the lifts of the paths used to construct formal capping surfaces for $a_1$ and $b_2$. }\label{fig:brbigon}
\end{center}
\end{figure}

\begin{lem}\label{lem:abndry}  If $a_k$ is a new generator coming from $\Gamma^+$, then $b_{\alpha-k}$ is a summand of $\partial^+a_k$. 
\end{lem}

\begin{proof}
Lemma~\ref{lem:abdefect} implies that $b_{\alpha-k}$ appears in the the boundary of $a_k$ with some coefficient; we need to prove that this coefficient is $1$.  With formal capping surfaces of the type described above,  Definition~\ref{def:gengr} and Equation~\ref{eq:rotation} together imply that $|a_k|-|b_{\alpha-k}|=1$.  Thus, the coefficient of $b_{\alpha-k}$ in the boundary of $a_k$ lies in $\zz_2$.  The disc discussed in Lemma~\ref{lem:abdefect} is the only disc in $\Delta(a_k; b_{\alpha-k})$ whose image lies inside the isotopy region, and the arbitrarily small difference in the lengths of $a_k$ and $b_{\alpha-k}$ implies that there cannot be a disc in $\Delta(a_k; b_{\alpha-k})$ whose image leaves the isotopy region.  Thus, the coefficient is $1$, not $0$.
\end{proof}

As a graded algebra,  $(\mathcal{A}', \partial')$ is isomorphic to  $(\mathcal{A}^+, \partial^+)$.  As a vector space, $\mathcal{A}'$ decomposes as $\mathcal{A}^-\oplus \mathcal{I}_{\mathcal{E}}$, where $\mathcal{I}_{\mathcal{E}}$ is the two-sided ideal generated by elements in the stabilizing pairs $\mathcal{E}_i$. Define $\tau: \mathcal{A}'\rightarrow \mathcal{A}'$ to be projection to $\mathcal{A}^-$, and define $F: \mathcal{A}'\rightarrow \mathcal{A}'$ by
\[
F(x) = \begin{cases}  ye_{1}^iz &  \text{if} \ x=ye_{2}^iz \text{ and }y\in \mathcal{A}^-\\ 
0 & otherwise.\end{cases} 
\]

\begin{lem}
\ $\tau$ satisfies
\begin{equation}\label{lem:tau}
\tau +id_{\mathcal{A}'}=F \circ \partial ' + \partial' \circ F.
\end{equation}
\end{lem} 
The proof is a straightforward computation.

\subsection{Step 2: The relationship between $\partial^-$ and $\partial^+$}\label{sec:step2}

In order to adjust the algebra isomophism between $\mathcal{A}'$ and $\mathcal{A}^+$ to an isomorphism of differential algebras, we compare the boundary maps $\partial^+$ and $\partial'$.   Terms in $\partial'$ come from two sources: the internal differentials on the stabilizing generators $\mathcal{E}_i$, and the differential $\partial^-$ which counts discs in $\Gamma^-$.  Lemma~\ref{lem:abndry} showed that the first type of term has a natural analogue in $\partial^+$, and in this step we study discs of the second type.

\begin{defn}\label{def:partition}
 For each $k\geq2$, let $\mathcal{I}(k)$ denote the set of ordered partitions $I=\{i_1, i_2,
...i_{|I|}\}$ of $k$ which satisfy 
\[0\leq|\ell(b_{\alpha-k})-\sum_{i_j\in I_i}
  \ell(a_{i_j})|\leq 2|I|\epsilon'.\]
  for $\epsilon'$ less than the absolute value of the total curvature of the isotopy region.
\end{defn}

For a generator $x$ of $\alg$, define a \dfn{special $x$-set} as
\begin{enumerate}
\item a term $y_1 b_{\alpha-k} y_2$ in $\partial^+x$; together with 
\item a collection of words $\{ w_{i_1}, ...w_{i_j}\}$ indexed by $I\in \mathcal{I}(k)$ with the property that each $w_{i_j}$ is a term in $\partial^+ a_{i_j}$; and such that
\item each $w_{i_j}$ and $y_l$ is written entirely in generators which come from crossings in $\Gamma^-$.
\end{enumerate}

\begin{lem}\label{lem:xset} There is a bijection between special $x$-sets and boundary discs in $\partial^- x$ with the properties that $u(\partial \dd)$ covers the isotoping strand of $\Gamma$ at least twice and $|u^{-1}(e)|=1$. 
\end{lem}

Before proving the lemma, we briefly consider the cases in which $u(\partial \dd)$ covers the isotoping strand once or $|u^{-1}(e)|=0$.  As above, let $S_u$ denote the formal capping surface associated to the image of $u$.  The post-isotopy image of $S_u$, as shown in Figure~\ref{fig:fcsisotopy}, is a formal capping surface which describes a boundary disc in $\partial^+ x$.    Lemma~\ref{lem:xset} describes how the isotopy affects maps $u$ when this simple analysis is insufficient. To understand the case when $|u^{-1}(e)|>1$, restrict $u$ to a subdisc of $\dd$ and apply Lemma~\ref{lem:xset}.

\begin{proof} 
For ease of visualization, we temporarily lift $\Gamma^-$ and $\Gamma^+$ to their $\alpha$-fold cyclic covers, and Figure~\ref{fig:gluing} suggests the idea behind the proof: a single disc representing a term in $\partial^- x$ corresponds to a collection of discs in $(\Sigma, \Gamma^+)$ which represent boundary terms in $\partial^+x$ and $\partial^+ a_{i_j}$. 

  \begin{figure}[ht] 
  \begin{center}
     \scalebox{.48}{\includegraphics{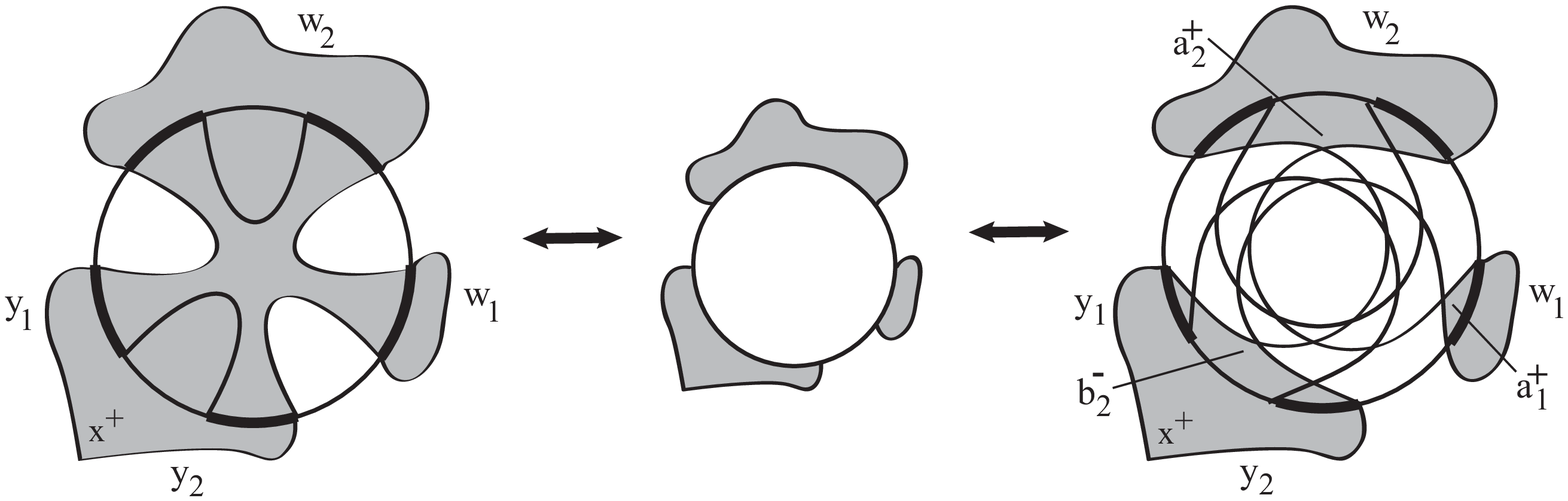}}
     \caption{\ }\label{fig:gluing}
           \end{center} 
 \end{figure}

First, consider the image of a map representing a term in $\df^-x$ and suppose that $u(\partial \dd)$ covers the isotoping strand $n>1$ times.  By Proposition~\ref{prop:defect-obstruction}, $u(\dd)$ lifts to a disc in the manifold $M$.  Removing the isotopy neighborhood cuts this lift into a collection of $n$ discs, each with a segment of its boundary lying on the boundary of the isotopy region.  The Lagrangian projection of this segment consists of a collection of arcs whose images were contiguous to $e$ in $\Gamma^-$ alternating with arcs whose images were separated from $e$ by $\Gamma^-$.  (The contiguous arc are shown in bold in Figure~\ref{fig:gluing}.)

Now replace the removed ball with its post-isotopy image.  The projection of each of the $n$ discs can be extended to the image of an admissible disc without singular points in $(\Sigma, \Gamma^+)$ by attaching a unique triangle along an arc in the boundary of the isotopy region; the new corner in the image of each disc is labeled either by $b_{\alpha-k}^-$ if the disc also contains $x^+$, or by some $a_{i_j}^+$. The labeling convention implies that the boundary of an $a_{i_j}$ triangle covers $i_j$ contiguous arcs, and the boundary of the $b_{\alpha-k}$ triangle covers $\alpha-k$ contiguous arcs.  Since each of the $\alpha$ contiguous arcs is covered by some triangle, it follows that the sum of the ${i_j}$ is $k$.  Furthermore, the total curvature associated to the isotopy region may be bounded close to zero, so Lemma~\ref{lem:stokes} ensures that Definition~\ref{def:partition} is satisfied. Thus, the new collection of discs represent terms forming a special $x$-set.

Given a special $x$-set, consider the images of the corresponding discs in the labeled diagram.  Remove a neighborhood of the image of the isotopy region, truncating each disc.  Replacing this region by its pre-isotopy image, there is a unique way to assign multiplicities to regions of $\Sigma\setminus \Gamma^-$ which matches the edges of the truncated discs. This defines an admissible map without singular points into $(\Sigma, \Gamma)$; since the isotopy region may be assumed arbitrarily small, Definition~\ref{def:partition} implies that the defect of this map is $0$ and therefore represents a term in $\partial^-x$.
\end{proof}

\subsection{Step 3: Constructing $s:\mathcal{A}^+\rightarrow \mathcal{A}'$}\label{sec:step3}

The map $s:\mathcal{A}^+\rightarrow \mathcal{A}'$ is defined inductively, based on the lengths of the generators.  Let $\mathcal{A}_{[0]}$ denote the subalgebra generated by the generators of $\mathcal{A}^+$ with length less than or equal to $\frac{2\pi}{\alpha}+\epsilon'$.  Ordering the remaining generators by increasing length,  let $\mathcal{A}_{[j]}$ denote the subalgebra generated by the elements in $\mathcal{A}_{[0]}$, together with the next $j$ generators.

\begin{lem}\label{lem:length} If $x_k$ is in $\mathcal{A}_{[j]}\backslash \mathcal{A}_{[j-1]}$ for $j\geq1$, then $\partial^+(x_k) \subset \mathcal{A}_{[j-1]}$.
\end{lem}

For each $a_{k}$, write $\partial^+a_{k}=b_{\alpha-k}+v_{k}$.  
Note that if $a_k\in \mathcal{A}_{[0]}$, then $b_{\alpha-k}$ is the only word in $\partial^+a_{k}$ which does not involve some generator coming from $\Gamma^-$.   

Define  $s_0:\mathcal{A}_{[0]} \rightarrow \mathcal{A}'$ by
\[
s_0(x) = \begin{cases}  
 x& \text{if }x \text{  is from  }\Gamma^-\\
e_{1}^k &  \text{if} \ x=a_{k} \\ 
e_{2}^k + v_{k} & \text{if}\ x=b_{\alpha-k}. \end{cases} 
\]

For the inductive step, suppose that $s_j:\mathcal{A}_{[j]}\rightarrow\mathcal{A}'$ is defined for $j\in \{ 0, 1, ...  M-1\}$.  Define:
\begin{equation}
s_M(x) = \begin{cases}  
 x& \text{if }x \text{ from } \Gamma^-\\
e_{1}^k +s_{M-1}\big(\sum  a_{i_1}\partial^+ a_{i_2}... \partial^+a_{i_{|I_i|}}\big)& \text{if} \ x=a_{k} \\ 
e_{2}^k + s_{M-1}(v_{k})+ &\\
\ \ \ s_{M-1}\big(\sum \partial^+a_{i_1}\partial^+ a_{i_2}... \partial^+a_{|I_i|}\big)& \text{if}\ x=b_{\alpha-k}. \end{cases} 
\end{equation}
Here, the sums are taken over $I_i\in \mathcal{I}(k)$ with $|I_i|>1$. 

Since $\alg$ is finitely generated, this process terminates after finitely many steps.  Call the resulting map $s$.

The length bounds coming from the definition of $\mathcal{I}(k)$ imply that the $s$-images of the sums are well-defined. Lemma~\ref{lem:abndry} implies that the coefficient of $b_{\alpha-k}$ is $1$, but the summands $v_k$  may have nontrivial coefficients.  In order for $s$ to preserve gradings, we multiply each term of  of the form $s(a_{i_1}\partial^+a_{i_2}...\partial^+a_{i_{|I_i|}})$ by an additional power of $t$.  Multiplying the  corresponding term in $s(b_{\alpha-k})$ by the same power of $t$ ensures $s$ preserves gradings.  However, we  suppress this coefficient for notational simplicity.

\subsection{Step 4: Relating $\partial'$ and $\partial^+$}\label{sec:tau}

We use the map defined in the previous step to define a new differential on $\mathcal{A}'$.  Let $\hat{\partial}:\mathcal{A}'\rightarrow \mathcal{A}'$ be defined by $\hat{\partial}=s\circ \partial^+\circ s^{-1}$.

Recall that $\tau$ is the projection of $\mathcal{A}'\cong \mathcal{A}^- \oplus \mathcal{I}_{\mathcal{E}}$ onto the first summand. 

\begin{lem}\label{lem:taucompeq}
On $\coprod_{i=1}^{\alpha-1}\mathcal{E}_i,  \partial'= \hat{\partial}.$
\end{lem}

The proof is a computation.  Although the equality of Lemma~\ref{lem:taucompeq} does not hold in general, composing the two differentials with $\tau$ yields the following equality:

\begin{lem}\label{lem:taucomp}
$\tau \circ \partial'=\tau \circ \hat{\partial}.$
\end{lem}

\begin{proof}
Fix $x$ to  be a generator associated to a crossing in $\Gamma^-$.  To prove the lemma, we compare the words appearing in $\tau \circ \partial'(x)$ and in $\tau \circ \hat{\partial}(x)$; it suffices to show that any word with no generators in the $\mathcal{E}_i$ appears in both $\hat{\partial}(x)$ and $\partial'(x)$.  

Terms in $\partial'(x)$ come from one of the following types of admissible maps without singular points:
\begin{enumerate}
\item Maps whose image is disjoint from the exceptional point.
\item Maps whose image covers the exceptional point, and such that $u(\partial\dd)$ covers the isotoping strand zero or one times; 
\item Maps whose image covers the exceptional point,  and such that $u(\partial\dd)$ covers the isotoping strand two or more times.
\end{enumerate} 
Clearly, $\tau\circ \partial'(x)=\partial'(x)$.

Maps of the first two types correspond to terms in $\partial^+x$ which do not involve any of the new generators.  Note that  Lemma~\ref{lem:xset} relates maps of the third type to special $x$-sets.  

We compare these terms to terms in $\hat{\partial}(x)$.  Since $s^{-1}(x)=x$, terms in $\hat{\partial}x=s\circ \partial^+ \circ s^{-1}(x)$ are the $s$-images of terms in $\partial^+$.  These separate into two types:
\begin{enumerate}
\item\label{1a} Words involving none of the generators associated to new crossings in $\Gamma^+$;
\item\label{1b} Words involving at least one of the generators associated to new crossings in $\Gamma^+$;
\end{enumerate}

In the second list,  words with some $a_{k}$ but no $b_j$ from the isotoping region vanish under $\tau$.  We show this by induction on $k$, where the base case is provided by the observation that $s(a_1)\in \mathcal{I}_{\mathcal{E}}$.  Now suppose the claim holds for all $a_k$ with $k<j$.  Applying $s$ to $a_j$ yields  $e_1^j$ and a (possibly empty) sum of products, each of which begins with an $a_k$ term for  some $k<j$.  By hypothesis, each of these will vanish under $\tau\circ s$, which proves the inductive step. 

To see that the remaining terms agree, recall the $s$-image of $b_{\alpha-k}$:
\[s(b_{\alpha-k})=e_{2}^k+s(v_k) +s\big(\sum_{I_i\in \mathcal{I}(k)} \partial^+a_{i_1}...\partial^+a_{i_{|I_i|}}\big).\]

We consider each term in turn.  First, note that since $e_2^k$ corresponds to a  generator coming from some new crossing in $\Gamma^+$, it will vanish under $\tau$.     Recall that $\partial^+a_k= b_{\alpha-k}+v_k$; if a summand of $v_k$ contains a generator corresponding to a new crossing, then its $s$ image will similarly vanish under $\tau$.  This leaves us with summands  of $\partial^+a_k$ written only in generators which come from crossings in $\Gamma^-$.  Each of these discs may be glued to a boundary disc for $x$ with a corner at $b_{\alpha-k}^-$, and after smoothing this corresponds to a disc in $\partial^-x$ whose boundary covers the isotoping curve of $\Gamma$ twice.  The remaining terms, each of the form $ \partial^+a_{i_1}...\partial^+a_{i_{|I_i|}}$, can similarly be taken together with a  word of the form $\partial^+ x=y_1 b_{\alpha-k}y_2$ to form a special $x$-set.   Lemma~\ref{lem:xset} implies that the associated discs can be glued and smoothed to form a  disc contributing to $\partial^-x$ whose boundary covers the isotoping strand more than two times.   

\end{proof}

\subsection{Step 5: Constructing $g:\mathcal{A}'\rightarrow \mathcal{A}'$}\label{sec:g}

Following Chekanov, we construct $g:\mathcal{A}'\rightarrow \mathcal{A}'$ as a composition of maps $g_j$, each of which is an elementary isomorphism affecting only the generators in $\mathcal{A}_{[j]}\backslash \mathcal{A}_{[j-1]}$.   Furthermore, each  $g_j$ inductively defines a new boundary map $\partial_{[j]}$ on $\mathcal{A}'$  by conjugation:
\[
\partial_{[j]}=g_{j}\partial_{[j-1]}g_{j}^{-1}.
\]

Set  $\partial_{[0]}=\hat{\partial}$. For the inductive step, suppose that for $1\leq k \leq j-1$, the maps $g_k$ satisfy $\partial_{[k]}|_{A_{[k]}}=\partial'|_{A_{[k]}}$.  Define $g_j: \mathcal{A}'\rightarrow\mathcal{A}'$ by
\[
g_j(x) = \begin{cases} x+F(\partial'(x)+\partial_{[j-1]}(x)) & \text{if}\ x \in \mathcal{A}_j\backslash \mathcal{A}_{[j-1]}\\ 
x & \text{otherwise}. \end{cases} 
\]

\begin{lem}
Writing $g=g_n \circ g_{n-1}\circ ...\circ g_2 \circ g_1$, 
\[
\partial'=g\circ \hat{\partial} \circ g^{-1}=g\circ s\circ \partial^+\circ s^{-1}\circ g^{-1}.
\]
This establishes a tame isomorphism between $(\mathcal{A}^+,\partial^+)$ and $(\mathcal{A}', \partial')$.

\end{lem}

The proof relies on Lemmas~\ref{lem:tau}, \ref{lem:length}, and \ref{lem:taucomp} and follows verbatim from \cite{chv}.

\bibliographystyle{amsplain} 
\bibliography{main}

\end{document}

%% file: header.tex
\usepackage[all]{xy}
\usepackage[usenames]{color}

\newcommand{\df}{\ensuremath{\partial}}

\newcommand{\alg}{\ensuremath{\mathcal{A}}}

\newcommand{\ind}{\operatorname{Ind}}

\newcommand{\rr}{\ensuremath{\mathbb{R}}}
\newcommand{\zz}{\ensuremath{\mathbb{Z}}}
\newcommand{\qq}{\ensuremath{\mathbb{Q}}}

\newcommand{\dd}{\ensuremath{\mathbb{D}}}

\theoremstyle{plain}
\newtheorem{thm}{Theorem}[section]
\newtheorem{cor}[thm]{Corollary}
\newtheorem{lem}[thm]{Lemma}

\newtheorem{prop}[thm]{Proposition}

\theoremstyle{definition}
\newtheorem{defn}[thm]{Definition}

\theoremstyle{remark}
\newtheorem{rem}[thm]{Remark}

\numberwithin{equation}{section}


%% file: lch4sfs-11-2.bbl
\def\cprime{$'$} \def\polhk#1{\setbox0=\hbox{#1}{\ooalign{\hidewidth
  \lower1.5ex\hbox{`}\hidewidth\crcr\unhbox0}}}
\providecommand{\bysame}{\leavevmode\hbox to3em{\hrulefill}\thinspace}
\providecommand{\MR}{\relax\ifhmode\unskip\space\fi MR }
\providecommand{\MRhref}[2]{%
  \href{http://www.ams.org/mathscinet-getitem?mr=#1}{#2}
}
\providecommand{\href}[2]{#2}
\begin{thebibliography}{10}

\bibitem{be:rational-tb}
K.~Baker and J.~Etnyre, \emph{Rational linking and contact geometry}, Preprint
  available as arXiv:0901:0380, 2009.

\bibitem{bg:leg-lens}
K.~Baker and J.~E. Grigsby, \emph{Grid diagrams and {L}egendrian lens space
  links}, J. Symplectic Geom. \textbf{7} (2009), no.~4, 415--448.

\bibitem{bgh:lens-HFK}
K.~L. Baker, J.~E. Grigsby, and M.~Hedden, \emph{Grid diagrams for lens spaces
  and combinatorial knot {F}loer homology}, Int. Math. Res. Not. IMRN (2008),
  no.~10, Art. ID rnm024, 39.

\bibitem{chv}
Yu. Chekanov, \emph{Differential algebra of {L}egendrian links}, Invent. Math.
  \textbf{150} (2002), 441--483.

\bibitem{cr:orbi-gw}
W.~Chen and Y.~Ruan, \emph{Orbifold {G}romov-{W}itten theory}, Orbifolds in
  mathematics and physics ({M}adison, {WI}, 2001), Contemp. Math., vol. 310,
  Amer. Math. Soc., Providence, RI, 2002, pp.~25--85.

\bibitem{products}
G.~Civan et~al., \emph{Product structures for {L}egendrian contact homology},
  To Appear in Math. Proc. Camb. Phil. Soc.

\bibitem{cornwell:lens-invts}
C.~Cornwell, \emph{Bennequin type inequalities in lens spaces}, Preprint
  available as arXiv:1002.1546v2, 2010.

\bibitem{ees:high-d-geometry}
T.~Ekholm, J.~Etnyre, and M.~Sullivan, \emph{Non-isotopic {L}egendrian
  submanifolds in {$\mathbb R\sp {2n+1}$}}, J. Differential Geom. \textbf{71}
  (2005), no.~1, 85--128.

\bibitem{ees:pxr}
\bysame, \emph{Legendrian contact homology in {$P\times\mathbb{R}$}}, Trans.
  Amer. Math. Soc. \textbf{359} (2007), no.~7, 3301--3335 (electronic).

\bibitem{yasha:icm}
Ya. Eliashberg, \emph{Invariants in contact topology}, Proceedings of the
  International Congress of Mathematicians, Vol. II (Berlin, 1998), no. Extra
  Vol. II, 1998, pp.~327--338 (electronic).

\bibitem{egh}
Ya. Eliashberg, A.~Givental, and H.~Hofer, \emph{Introduction to symplectic
  field theory}, Geom. Funct. Anal. (2000), no.~Special Volume, Part II,
  560--673.

\bibitem{etnyre:intro}
J.~Etnyre, \emph{Introductory lectures on contact geometry}, Topology and
  geometry of manifolds (Athens, GA, 2001), Proc. Sympos. Pure Math., vol.~71,
  Amer. Math. Soc., Providence, RI, 2003, pp.~81--107.

\bibitem{ens}
J.~Etnyre, L.~Ng, and J.~Sabloff, \emph{Invariants of {L}egendrian knots and
  coherent orientations}, J. Symplectic Geom. \textbf{1} (2002), no.~2,
  321--367.

\bibitem{geiges:intro}
H.~Geiges, \emph{An introduction to contact topology}, Cambridge Studies in
  Advanced Mathematics, vol. 109, Cambridge University Press, Cambridge, 2008.

\bibitem{kt:cr-seifert}
Y.~Kamishima and T.~Tsuboi, \emph{C{R}-structures on {S}eifert manifolds},
  Invent. Math. \textbf{104} (1991), no.~1, 149--163.

\bibitem{joan:leg-in-lens}
J.~Licata, \emph{Invariants for legendrian knots in lens spaces}, To Appear in
  Comm. Contemp. Math.

\bibitem{joan:grid-1-comp}
\bysame, \emph{Legendrian grid number one knots and augmentations of their
  differential algebras}, To appear in the Proceedings of the Heidelberg Knot
  Theory Semester.

\bibitem{ls:tb-paper}
J.E. Licata and J.M. Sabloff, \emph{Seifert surfaces in {S}eifert fiber spaces:
  constructions and computations}, In preparation.

\bibitem{lisca-matic:transverse}
P.~Lisca and G.~Mati{\'c}, \emph{Transverse contact structures on {S}eifert
  3-manifolds}, Algebr. Geom. Topol. \textbf{4} (2004), 1125--1144
  (electronic).

\bibitem{lutz}
R.~Lutz, \emph{Structures de contact sur les fibres principaux en cercles de
  dimension trois}, Ann. Inst. Fourier, Grenoble \textbf{27} (1977), no.~3,
  1--15.

\bibitem{massot}
P.~Massot, \emph{Geodesible contact structures on 3-manifolds}, Geom. Topol.
  \textbf{12} (2008), no.~3, 1729--1776.

\bibitem{moerdijk-pronk}
I.~Moerdijk and D.~A. Pronk, \emph{Simplicial cohomology of orbifolds}, Indag.
  Math. (N.S.) \textbf{10} (1999), no.~2, 269--293.

\bibitem{lenny:computable}
L.~Ng, \emph{Computable {L}egendrian invariants}, Topology \textbf{42} (2003),
  no.~1, 55--82.

\bibitem{lenny:lsft}
\bysame, \emph{Rational {S}ymplectic {F}ield {T}heory for {L}egendrian knots},
  Invent. Math. (To Appear).

\bibitem{ns:augm-rulings}
L.~Ng and J.~Sabloff, \emph{The correspondence between augmentations and
  rulings for {L}egendrian knots}, Pacific J. Math. \textbf{224} (2006), no.~1,
  141--150.

\bibitem{lenny-lisa}
L.~Ng and L.~Traynor, \emph{Legendrian solid-torus links}, J. Symplectic Geom.
  \textbf{2} (2004), no.~3, 411--443.

\bibitem{orlik}
P.~Orlik, \emph{Seifert manifolds}, Lecture Notes in Mathematics, Vol. 291,
  Springer-Verlag, Berlin, 1972.

\bibitem{ozturk}
F.~{\"O}zt{\"u}rk, \emph{Generalised {T}hurston-{B}ennequin invariants for real
  algebraic surface singularities}, Manuscripta Math. \textbf{117} (2005),
  no.~3, 273--298.

\bibitem{s1bundles}
J.~Sabloff, \emph{Invariants of {L}egendrian knots in circle bundles}, Comm.
  Contemp. Math. \textbf{5} (2003), no.~4, 569--627.

\bibitem{duality}
\bysame, \emph{Duality for {L}egendrian contact homology}, Geom. Topol.
  \textbf{10} (2006), 2351--2381 (electronic).

\bibitem{seaton:thesis}
C.~Seaton, \emph{Two {G}auss-{B}onnet and {P}oincar\'e-{H}opf theorems for
  orbifolds with boundary}, Ph.D. thesis, University of Colorado, 2004.

\bibitem{seaton:two-thms}
\bysame, \emph{Two {G}auss-{B}onnet and {P}oincar\'e-{H}opf theorems for
  orbifolds with boundary}, Differential Geom. Appl. \textbf{26} (2008), no.~1,
  42--51.

\bibitem{sha:secondary-euler}
J.P. Sha, \emph{A secondary {C}hern-{E}uler class}, Ann. of Math. (2)
  \textbf{150} (1999), no.~3, 1151--1158.

\end{thebibliography}
